\documentclass[12pt]{amsart}
\usepackage{graphicx}
\usepackage{blkarray}
\usepackage{amsmath}
\usepackage{tabu}
\usepackage{marginnote}
\usepackage[top=1cm, bottom=1cm, outer=1cm, inner=1cm, heightrounded, marginparwidth=1cm, marginparsep=1cm]{geometry}
\usepackage{hhline}
\usepackage{amsthm}
\usepackage{amssymb}
\usepackage{fullpage}
\usepackage{multirow}
\usepackage[all]{xy}
\usepackage{longtable}
\usepackage{arydshln}
\usepackage{lscape}
\usepackage{rotating}
\usepackage{longtable}
\usepackage{float}
\usepackage{bm}
\usepackage{float}
\usepackage[makeroom]{cancel}
\usepackage{nameref}
\usepackage{tikz}
\usetikzlibrary{arrows,decorations.markings}
\usepackage{setspace}
\usepackage{multirow}
\usepackage{appendix}
\usepackage[makeroom]{cancel}
\usepackage{color}
\usepackage{chngcntr}
\usepackage{wrapfig}
\usepackage{fullpage}
\usepackage{caption}
\usepackage{subcaption}
\usepackage{verbatim}
\usepackage{tikz}
\usetikzlibrary{arrows,decorations.markings}
\usetikzlibrary{shapes}
\usetikzlibrary{matrix}
\usetikzlibrary{graphs}
\usepackage{arydshln}
\usetikzlibrary{backgrounds}
\usepackage{amsmath, amssymb}
\usepackage{amssymb}
\usepackage{fullpage}
\counterwithin{table}{section}
\usepackage{longtable}
\usepackage{caption}
\usepackage[makeroom]{cancel}
\usepackage{multirow}
\usepackage{tikz}
\usetikzlibrary{arrows,decorations.markings}
\usepackage[pdftex,
colorlinks,
linkcolor=blue,
urlcolor=blue,
citecolor=red
]{hyperref}
\usepackage{bookmark}
\newcommand{\teq}{\trianglelefteq}
\theoremstyle{plain}
\newtheorem{theorem}{Theorem}[section]
\newtheorem{lemma}[theorem]{Lemma}
\newtheorem*{tho}{Main Theorem}
\newtheorem{prop}[theorem]{Proposition}
\newtheorem{cor}[theorem]{Corollary}
\newtheorem{lem}[theorem]{Lemma}
\theoremstyle{definition}
\newtheorem{definition}[theorem]{Definition}
\newtheorem{example}[theorem]{Example}

\theoremstyle{remark}
\newtheorem{remark}[theorem]{Remark}
\numberwithin{equation}{section}
\def\sub{\subseteq}

\def\C{\mathbb C}
\def\F{\mathbb F}

\def\Z{\mathbb Z}
\def\bB{\mathbf B}
\def\bG{\mathbf G}

\def\bT{\mathbf T}
\def\bU{\mathbf U}

\def\cF{\mathcal F}
\def\cA{\mathcal A}

\def\cC{\mathcal C}
\def\cD{\mathcal D}

\def\cH{\mathcal H}
\def\cI{\mathcal I}
\def\cJ{\mathcal J}
\def\cK{\mathcal K}
\def\cL{\mathcal L}

\def\cP{\mathcal P}
\def\cR{\mathcal R}
\def\cS{\mathcal S}

\def\cZ{\mathcal Z}

\def\rA{\mathrm A}
\def\rB{\mathrm B}
\def\rC{\mathrm C}
\def\rD{\mathrm D}
\def\rE{\mathrm E}
\def\rF{\mathrm F}
\def\rG{\mathrm G}

\def\rU{\mathrm U}
\def\rY{\mathrm Y}
\def\fC{\mathfrak C}
\def\fO{\mathfrak O}

\def\tx{\tilde x}

\def\tX{\tilde X}
\def\tY{\tilde Y}

\def\tlm{\tilde \lambda}
\def\ua{{\underline a}}
\def\ub{{\underline b}}

\def\uhb{\hat{\underline b}}
\def\ui{{\underline i}}

\def\Ind{\operatorname{Ind}}

\def\Inf{\operatorname{Inf}}
\def\Irr{\operatorname{Irr}}

\def\Tr{\operatorname{Tr}}
\def\lm{\lambda}
\def\hlm{\hat \lambda}

\title[Irreducible characters of $\textrm{UD}_6(q)$ and $\textrm{UE}_6(q)$]{The irreducible characters of 
the Sylow $p$-subgroups of the Chevalley groups $\mathrm{D}_6(p^f)$ and $\mathrm{E}_6(p^f)$}

\author{Tung Le, Kay Magaard and Alessandro Paolini}

\address{} \email{kymgrd@yahoo.com}

\address{Department of Mathematics and Applied Mathematics, University of Pretoria, Pretoria 0002, South Africa}
\email{lttung96@yahoo.com}

\address{FB Mathematik, TU Kaiserslautern, Postfach 3049, 67653 Kaiserslautern, Germany.} \email{paolini@mathematik.uni-kl.de}

\thanks{Date: \today. 
}

\begin{document}

\begin{abstract} We parametrize the set of irreducible characters of 
the Sylow $p$-subgroups of the Chevalley groups $\rD_6(q)$ and $\rE_6(q)$, for an 
arbitrary power $q$ of any prime $p$. 
In particular, we establish that the parametrization 
is uniform for $p \ge 3$ in type $\rD_6$ and for $p \ge 5$ in type $\rE_6$, while 
the prime $2$ in type $\rD_6$ and the primes $2,$ $3$ in type $\rE_6$ yield character degrees of the form 
$q^m/p^i$ which force a departure from the generic situations.  
Also for the first time in our analysis we see a family of 
irreducible characters of a classical group 
of degree $q^m/p^i$ where $i > 1$ which occurs in type $\rD_6$.
\end{abstract}

\maketitle

\section{Introduction}\label{sec:intr} 

Let $q$ be a power of a prime $p$, and let $G$ be a finite group of Lie type over $\mathbb{F}_q$ and 
$U$ be a
Sylow $p$-subgroup of $G$. 
Let $B:=N_G(U)$ and let 
$P$ be a parabolic subgroup of $G$ with $U < B \leq P \leq G$. 
We denote by $\ell$ a prime distinct from $p$. 

Harish-Chandra theory for $\ell$-modular representations of general finite groups of Lie type initiated by Hiss in \cite{His90,His93} 
and continued in \cite{GHM96}. The theory suggests that the representation theory of parabolic subgroups $P$ of $G$ as above has strong influence on the representation theory of $G$, 
in particular towards a determination of its decomposition numbers.  
This is further evidenced by work of Gruber and Hiss \cite{GH97} for classical groups at primes $\ell$.
More recently this has been taken further in \cite{DM15} and \cite{DM16} for some non-linear classes of primes. 
For small rank groups the calculation of decomposition numbers in \cite{Him11}, \cite{HN14} and \cite{HN15} 
made strong use of the representation theory of parabolic subgroups along with induction/restriction methods to 
compute decomposition numbers. Most recently the third author \cite{Pao17} was able to compute most decomposition numbers
of the groups $\rD_4(2^f)$ using the generic character table of $\rU\rD_4(2^f)$ which was computed in \cite{GLM17}; here and throughout 
$\rU\rY_r(q)$ denotes a Sylow $p$-subgroup of the 
group $\rY_r(q)$ of type $\rY$ and rank $r$ defined over $\F_q$. 

Our present work is a first step in computing the character tables of $\mathrm{UD}_6(q)$ and $\mathrm{UE}_6(q)$ which 
are intended to aid with the determination of the decomposition numbers of all finite groups of Lie type up to rank $8$, in particular of those of exceptional type. 
The elements of $\Irr(U)$ where $G \cong \rF_4(q)$, with $q$ odd, were parametrized in \cite{GLMP16} by means 
of a recursive procedure, relying on basic character correspondences,  which leads to a natural construction of 
characters via induction from linear characters of certain subgroups. Indeed the terminal points of our algorithm are certain 
subquotients of $U$, which we call cores.
To construct the members of a family one starts with a core $Q$ with centre $Z$, and with a subset $\Irr(Q)_Z$ of $\Irr(Q)$ 
consisting of characters which lie over the centre in such a way as to not contain any root subgroup of $Z$ in their kernels. For each element
of $\Irr(Q)_Z$ one can trace back through the algorithm, and write down a unique character of $U$. 

If $Q = Z$ is abelian, then we can simply extract the family and its parameters. We have been able to 
completely automate this part of the algorithm. If $Q \neq Z$, then we need to determine $\Irr(Q)_Z$, which in the situation 
of ${\rm UF}_4(q)$ involved a manageable amount of case analysis. While the number of nonabelian cores for 
${\rm UF}_4(q)$ is $6$, this number increases to $105$ for ${\rm UE}_6(q)$ and to several millions for ${\rm UE}_8(q)$.
We recall from \cite{GLM17} and \cite{HLM16} that the characters of ${\rm UE}_6(q)$ are naturally partitioned into 
$833$ families which are indexed by the antichains in the poset of positive roots. To each family we apply  
our algorithm, which naturally splits each family into collections of subfamilies. For type $\rE_8$ the first 
partition already leads to $25080$ families. 
Thus it becomes clear that the generic character tables of the groups ${\rm UY}_r(q)$ are best processed in a machine-readable format and 
ideally in a format that can be incorporated into \cite{CHEVIE}, the computer algebra system which provides a 
platform for calculations with the generic character tables of finite groups of Lie type. 
Our main theorem thus takes the following form:

\begin{tho}
Let $q$ be a power of a prime $p$, let $G$ be a finite Chevalley group over $\mathbb{F}_q$ of type $\rD_6$ or $\rE_6$, and let $U$ be a Sylow $p$-subgroup of $G$. 
Then the irreducible characters of $U$ are completely parametrized. Each character can be obtained as an induced character of a 
linear character of a certain determined subgroup. In 
particular, if $v:=q-1$, we have 
\begin{enumerate}
\item $|\Irr(\rU\rD_6(q))| =
\begin{cases}
p_1(v), & \text{ if } q \text{ is odd}, \\
p_1(v)+3v^4(v^4+18v^3+63v^2+58v+9), & \text{ if } 
q=2^f,
\end{cases} $ 
\vspace{1mm}
\item $|\Irr(\rU\rE_6(q))| =
\begin{cases}
p_2(v), & \text{ if } \gcd(q,6)=1, \\
p_2(v)+v^6(v^2+6v+12), & \text{ if } q=3^f, \\
p_2(v)+3v^4(2v^4+26v^3+103v^2+317v+45), & \text{ if }
q=2^f,
\end{cases} $
\end{enumerate}
where $p_1(v)$ and $p_2(v)$ are polynomial expressions in $v$ as in 
Tables \ref{tab:fam3D6} and \ref{tab:fam5E6} respectively. 
\end{tho}

We collect here further consequences of the parametrization in 
the above theorem. When $p$ is a bad prime for $\rE_6$, then $\Irr(\rU\rE_6(q))$ possesses characters of degree $q^i/2$ for $3 \leq i \leq 15$ if $p =2$ and 
of degree $q^7/3$ if $p =3$, whereas if $p=2$ then $\Irr(\rU\rD_6(q))$ possesses elements of degree $q^i/2$ for $3 \leq i \leq 11$, and 
also a family of characters of degree $q^{10}/4$ which is obtained by induction from the family $\cF_3^{4, p=2}$ in Table \ref{tab:coresD6} of characters of degree $q^6/4$; 
other irreducible character degrees 
are all powers of $q$. 
The numbers of characters of 
fixed degree are given in Tables \ref{tab:fam3D6} 
to \ref{tab:fam2E6}. We easily 
check in these cases the validity of the generalization of 
\cite[Conjecture B]{Is07} in types different from $\rA$, 
which in turn refines \cite[Conjecture 6.3]{Leh74}, namely the numbers of 
irreducible characters of $U$ of fixed degree 
can always be expressed as polynomial expressions 
in $v$ with non-negative integral coefficients if 
$p$ is good. 
One also immediately deduces by the records 
in Table \ref{tab:fam3E6} for $\Irr(\rU\rE_6(3^f))$
that an extension of such statement to bad primes would not hold; 
this is the only such instance 
for the groups $\rU\rY_r(q)$ with $\rY$ of simply laced type and $r\le 6$, 
namely in this case $\rU\rY_r(q)$ is a natural quotient 
of $U$, and 
a parametrization for $\Irr(\rU\rY_r(q))$ is 
obtained via the labels determined in our theorem. 
The actual complete list of families is available on the webpage of the third author \cite{LMP} 
in both tabular and machine-readable format. 

The obstruction to automating the parametrization of $\Irr(U)$ is the nonabelian cores mentioned above. Thus our focus in this paper 
is on nonabelian cores with a view towards automating these calculations as well. The total number of families of nonabelian cores that 
we have to consider is $27$ for $\rD_6$ and $105$ for $\rE_6$. Fortunately several cores are isomorphic, which reduces our problem to $7$ isomorphism types of cores for $\rD_6$ and $16$ for 
$\rE_6$ which are easily separated by a set of three invariants; this is proved in Section \ref{sec:nonc}. 
Also certain cores are isomorphic to ones that we have seen in \cite{GLMP16}, which 
simplifies our work even further. In Section \ref{sec:cors} we begin by proving a variant of our reduction lemma which serves as a 
foundational tool of our analysis of nonabelian cores. Also in this section we introduce the concept of a generalized root group which 
allows us to consider transversals which are well suited for our character correspondences, and the concept of a ``circle quattern''. 
In fact it is the latter concept which we believe will be crucial in automating the analysis of nonabelian cores. We illustrate 
all of this in our analysis of the nonabelian cores of $\rU\rD_6(q)$ and $\rU\rE_6(q)$ in Section \ref{sec:para}. We record the results 
of our analysis of the nonabelian cores in Tables \ref{tab:coresD6} and \ref{tab:coresE6}. 

To finish, we remark that for groups of rank higher than $6$ 
the three invariants mentioned above are not strong enough to separate cores into isomorphism types, and we illustrate this with an example in 
${\rm UE}_7(q)$. Also we remark here that for ${\rm UE}_8(q)$ the 
cardinality of the set of invariants of nonabelian cores is in the neighborhood 
of $2\cdot 10^5$, and that the number of isomorphism types is around $4\cdot 10^5$; again making clear the need for automation. 

\vspace{1mm}

\noindent
\textbf{Acknowledgement:} The authors would like to thank S. M. Goodwin for 
helpful remarks about the work, and G. Malle for his valuable comments on an 
earlier version of the paper. 
Additional thanks go to E. O'Brien for confirming our results independently 
for small values of $q$. 
Part of the work was developed during three research 
visits, namely of the three authors at the EPFL in December 2016, 
of the first author at the University of Birmingham in June 2017, and 
of the first and the second author at the TU Kaiserslautern in July 2017, which 
were funded respectively by the Centre Interfacultaire Bernoulli, 
by the SFB-TRR 195 and by the NRF Incentive Grant; special thanks 
go to the three institutes for the funding support and the kind hospitality. 
The third author acknowledges financial support from the School of Mathematics of the University of Birmingham, 
the ERC Advanced Grant 291512, and the SFB-TRR 195.

\section{Preliminaries}\label{sec:gens}

\subsection{Characters of finite groups.} \label{ss:characters}

Let $G$ be a finite group. For $g, h \in G$, we use the notation $g^h:=h^{-1}gh$ (respectively ${}^hg:=hgh^{-1}$) for 
the right (respectively left) conjugation in $G$. The centre 
of $G$ is denoted by $Z(G)$. We usually denote by $\chi$ an irreducible 
character afforded by some representation. 
All characters considered in this work are 
ordinary; we refer to \cite{Is} for the basics on character theory of finite groups. 
We denote by $\ker(\chi)$ the kernel of a character $\chi$, and by $Z(\chi)$ the 
centre of $\chi$. Moreover, we denote by $\Irr(G)$ the set of irreducible characters of $G$. 

If $N \teq G$, and $\chi \in \Irr(G/N)$, we denote 
by $\Inf_N^G(\chi)$ 
the inflation of $\chi$ to $G$. 
For $H \le G$ and $\eta\in\Irr(H)$, we denote by $\Ind_H^G(\eta)$, or shortly $\eta^G$, the 
induction of the character $\eta$ from $H$ to $G$, and we define
$$\Irr(G \mid \eta):=\{\chi \in \Irr(G) \mid \langle \chi, \eta^G \rangle \ne 0\}=
\{ \chi \in \Irr(G) \mid \langle \chi|_{H}, \eta  \rangle \ne 0 \},$$
where $\langle \,\, , \, \rangle$ is the usual scalar product of characters. Let $\chi_1$ and $\chi_2$ be two characters of $G$. 
The character $\chi_1 \otimes \chi_2$ 
denotes the tensor product of $\chi_1$ and $\chi_2$. If $H \le G$, and $\chi \in \Irr(G)$ and 
$\psi \in \Irr(H)$, then $(\chi|_H \otimes \psi)^G=\chi \otimes \psi^G$. 
Moreover, if $N$ 
is a normal subgroup of $G$ contained in $H$, then for every $\chi \in \Irr(H/N)$ we have that 
$$\Inf_{G/N}^G \Ind_{H/N}^{G/N} \chi= \Ind_H^G \Inf_{H/N}^H \chi.$$

Let $\eta \in \Irr(N)$ with $N \teq G$. For $g\in G$, we denote by ${}^g\eta$ the irreducible character 
of $N$ such that ${}^g\eta(x):=\eta(x^g)$ for every $x \in N$. The group $G$ naturally acts 
on $\Irr(N)$ by conjugation. 
Let us define 
the inertia subgroup of $\eta$ in $G$ by $I_G(\eta):=\{g \in G \mid {}^g\eta=\eta\}$. Then 
$$\Ind_{I_G(\eta)}^G: \Irr(I_G(\eta) \mid \eta) \longrightarrow \Irr(G \mid \eta)$$
is a bijection of irreducible characters. Moreover, if $Z \le Z(G)$ is such that $Z \cap N=1$, 
and $\lambda \in \Irr(Z)$, then
$$\Inf_{G/N}^G: \Irr(G/N \mid \lambda) \longrightarrow \Irr(G \mid \Inf_{G/N}^G(\lambda))$$
is also a bijection. 

We finish by describing the set $\Irr(\F_q)$. Let us define $\phi: \F_q \rightarrow \mathbb{C}^\times$ 
by $\phi(t):=e^{2\pi i \Tr(t)/p}$ for all $t \in \F_q$, where $\Tr: \F_q \rightarrow \F_p\cong \Z_p$ is the 
trace map of the field extension $\F_q \mid \F_p$. For each $a\in\F_q$, we define the map $\phi_a$ by $\phi_a(t):=\phi(at)$ for every $t \in \F_q$. Then 
$\Irr(\F_q)=\{\phi_a \mid a \in \F_q\}$. Moreover, it is easy to see that if $a \in \F_q^{\times}$, then 
$$\ker(\phi_a)=\{a^{p-1}t^p-t \mid t \in \F_q\}.$$

\subsection{Simple algebraic groups and Frobenius morphisms.} \label{ss:reductive}

We refer to \cite{DM} and \cite{MT} for basic properties and definitions of finite reductive groups. 
Let $q:=p^f$, where $p$ is a prime and $f \in \Z_{>0}$. 
Let $\F_q$ be a general finite field with $q$ elements, and let 
$k:=\overline{\F}_p$ be an algebraic closure of $\F_q$. 
We denote by $\bG$ a simple 
algebraic group over the field $k$. 

Let $F: \bG \to \bG$ be a standard Frobenius morphism. 
Let $\bT$ be a maximal torus of $\bG$ such that $F(\bT)=\bT$, and let $\bB$ be a Borel subgroup of $\bG$ containing $\bT$ such that $F(\bB)=\bB$.
Let $\bU$ be the unipotent radical of $\bB$. Here, $\bB= N_\bG(\bU)=\bT \bU$. 
From now on, we fix such $F$-stable subgroups $\bT$, $\bU$ and $\bB$.
The group $G:=\bG^F$ of fixed points of $\bG$ under $F$ is a finite reductive group. 
Further, we set $B:=\bB^F$, $T:=\bT^F$, and $U:=\bU^F$.
Here, we have $B=N_G(U)=T \ltimes U$, and $U$ is a Sylow $p$-subgroup of $G$. 
All subgroups $B$, $T$, and $U$ are 
fixed for the rest of the work.  

Let $\Phi$ denote the root system associated to $\bG$ with respect to $\bT$, 
and let $\Pi:=\{\alpha_1, \dots, \alpha_r\}$ be the set of simple roots of $\Phi$, 
where $r$ is the rank of $\Phi$. 
Let $\Phi^+ \subseteq \Phi$ denote the set of positive roots in $\Phi$, and let $m$ be 
the number of positive roots. 
We fix a total ordering on $\Phi^+=\{\alpha_1, \dots, \alpha_{m}\}$ 
by refining the partial order on $\Phi^+$, defined 
by $\alpha < \beta$ if $\beta - \alpha$ 
is a sum of simple roots; this agrees with the ordering in \cite{GAP4}. If $\Phi$ is of type $\rY$ and rank $r$, we sometimes denote $U$ more explicitly by 
$\rU \rY_r(q)$. 

For each $\alpha \in \Phi^+$, there exist an $F$-stable subgroup $\bU_{\alpha} \subseteq \bU$ and an 
isomorphism $x_{\alpha}: k \rightarrow \bU_{\alpha}$, such that 
$$\bU_{\alpha}:=\{x_{\alpha}(t) \mid t \in k\} \cong (k,+), \qquad \text{and} 
\qquad X_{\alpha}:=\bU_{\alpha}^F=\{x_{\alpha}(t) \mid t \in \F_q\} \cong (\F_q,+).$$
The subgroup $X_{\alpha}$ of $G$ is called a \emph{root subgroup}, and an element 
of the form $x_{\alpha}(t)$ is called a \emph{root element}. 
We often abbreviate and write $X_i$ for $X_{\alpha_i}$ and $x_i$ for $x_{\alpha_i}$. The group 
$U$ is the product of all root subgroups labelled by positive roots, and each element of $U$ can be uniquely written as a product 
$x_1(t_1)\cdots x_{m}(t_{m})$ for some $t_1, \cdots, t_{m} \in \F_q$. 
A presentation for $U$ is given by the Chevalley relations
\begin{equation}\label{eq:comrel}
[x_{\alpha}(s), x_{\beta}(r)]=
\prod_{\substack{i, j \in \mathbb{Z}_{>0} \mid i\alpha+j\beta \in \Phi^+}}
x_{i\alpha+j\beta}(c_{i,j}^{\alpha,\beta}(-r)^js^i)
\end{equation}
for every $r, s \in \F_q$ and $\alpha, \beta \in \Phi^+$, and for some $c_{i,j}^{\alpha,\beta} \in \mathbb{Z}\setminus \{0\}$, 
called \emph{Lie structure constants}. As proved in \cite[Section 5.2]{Car}, the parametrizations of the root subgroups 
can be chosen so that the structure constants $c_{i,j}^{\alpha,\beta}$ are always 
$\pm 1$, $\pm 2$, $\pm 3$, where $\pm 2$ occurs only for $G$ of types $\rB_r$, $\rC_r$, $\rF_4$ or $\rG_2$, and $\pm 3$
only occurs for $G$ of type $\rG_2$. The signs are determined by fixing the ones corresponding to the so-called extraspecial pairs of roots; 
our choice agrees with the records in the computer algebra system \cite{MAGMA}.

\subsection{Quattern groups.} \label{ss:quattern} We now recall some 
properties that link the structure of $\Phi^+$ with that of $U$. If 
$\cA=\{\alpha_{i_1}, \dots, \alpha_{i_k}\}$ is a subset of $\Phi^+$ where $i_1 < \dots < i_k$, 
we define
$$X_{\cA}:=\prod_{j=1}^{k}X_{\alpha_{i_j}}.$$
This is in general not always a subgroup, but it will be in all cases of our interest. 

We recall some definitions and properties from \cite{HLM16}. We say that $\cP$ is 
a \emph{pattern} in $\Phi^+$ if $\alpha, \beta \in \cP$ and 
$\alpha + \beta \in \Phi^+$ imply $\alpha + \beta \in \cP$. 
Patterns are also known as \emph{closed subsets} of $\Phi^+$, see for example \cite[Definition 13.2]{MT}. 
It is easy to check, with no restrictions 
on the prime $p$, that if $\cP$ is a pattern, then $X_{\cP}$ is a subgroup of $U$. If $p$ is not a very 
bad prime for $\Phi^+$, the converse also holds. For very bad primes the converse does not hold in general. 
For example, if $p=2$ and $\alpha_2$ is the simple short root in type $\rB_2$, then $X_{\alpha_2}X_{\alpha_1+\alpha_2}$ is a subgroup of $\rU \rB _2(2^f)$, but $\{\alpha_2, \alpha_1+\alpha_2\}$ is not a pattern. 

Let $\cK, \cP$ be two patterns and $\cK \subseteq \cP$. We say that $\cK$ is \emph{normal} in $\cP$, 
denoted as $\cK \teq \cP$, if for all $\delta \in \cK$ and $\beta \in \cP$, $\delta+\beta \in \cP$ implies $\delta+\beta \in \cK$. 
If $\cK \teq \cP$, then we have that $X_{\cK} \teq X_{\cP}$. If $p$ is not 
a very bad prime for $\Phi^+$, then $X_{\cK} \teq X_{\cP}$ for two patterns $\cK \subseteq \cP$ 
also implies $\cK \teq \cP$. Again, this is not true in type $\rB_2$ when $p=2$, namely $X_{\alpha_1+\alpha_2} \teq 
\rU \rB_2 (2^f)=X_{\Phi^+}$, but $\{\alpha_1+\alpha_2\}$ is not a normal pattern in $\Phi^+$.

A subset $\cS \subseteq \Phi^+$ is called a \emph{quattern} if $\cS=\cP \setminus \cK$ for some 
pattern $\cP$ and $\cK \teq \cP$. 
Given a quattern $\cS \subseteq \Phi^+$ such that $\cS=\cP \setminus \cK$, 
we define the \emph{quattern group} $X_\cS$ associated to $\cS$ by
$$X_{\cS}:=X_{\cP} / X_{\cK}.$$
This subquotient of $U$ is well-defined, in the sense that if $\cS=\cP' \setminus \cK'$ for 
$\cP'$ a quattern and $\cK' \teq \cP'$, then $X_{\cS} \cong X_{\cS'}$. 

If $\cS$ is a quattern, we define 
$$\cZ(\cS): =\{\gamma \in \cS \mid \gamma+\alpha \notin \cS \text{ for all } \alpha \in \cS\}$$
the set of central roots in $\cS$, and 
$$\cD(\cS): = \{\gamma \in \cZ(\cS) \mid \alpha+\beta \ne \gamma \text{ for all } \alpha, \beta \in \cS\}$$
the set of roots parametrizing the root subgroups which are direct factors in $X_\cS$. We have 
$$Z(X_\cS) = X_{\cZ(\cS)} \qquad \text{and} \qquad X_\cS = X_{\cS \setminus \cD(\cS)} \times X_{\cD(\cS)}.$$
We define the set of irreducible characters of $X_{\cS}$ with central root support $\cZ \subseteq \cZ(\cS)$ by 
\begin{equation*}\label{eq:irr}
\Irr(X_\cS)_\cZ := \{\chi \in \Irr(X_\cS) \mid X_\alpha \not\sub \ker(\chi) \text{ for all } \alpha \in \cZ\}.
\end{equation*}
Hence we have
\begin{equation}\label{eq:decomp}
\Irr(X_\cS)_\cZ = \bigsqcup_{\lambda \in \Irr(X_{\cS})_{\cZ}} \Irr(X_\cS \mid \lambda), 
\end{equation}
and it is easy to see that 
\begin{equation}\label{eq:allfam}
\sum_{\chi \in \Irr(X_\cS)_\cZ}\chi(1)^2=q^{|\cS\setminus \cZ|}(q-1)^{|\cZ|}. 
\end{equation}

The importance of studying quatterns comes from the fact that we can partition $\Irr(U)$ 
into families of irreducible characters of quattern groups $X_{\cS}$ with central root support a certain 
$\cZ \subseteq \cZ(\cS)$. More precisely, let $\Sigma$ denote an antichain of $\Phi^+$, that is, 
a subset of $\Phi^+$ such that 
$$\alpha, \beta \in \Sigma \text{ and } \alpha\neq\beta \Longrightarrow \alpha \nleq \beta \text{ and } \beta \nleq \alpha.$$
The subset $\cK_{\Sigma}$ defined by
$$\cK_\Sigma:=\{\beta \in \Phi^+ \mid \beta \nleq \gamma \text{ for all } \gamma \in \Sigma \}$$
is a normal subset of $\Phi^+$. We define the \emph{standard quattern} $\cS_{\Sigma}$ associated to $\Sigma$ by 
$\cS_{\Sigma}:=\Phi^+ \setminus \cK_{\Sigma}$. Notice that $\Sigma=\cZ(\cS_{\Sigma})$. Finally, we define
$$\Irr(U)_{\Sigma}:=\{ \Inf_{X_{\cS_\Sigma}}^{U} (\chi) \mid \chi \in \Irr(X_{\cS_{\Sigma}})_{\Sigma}\}.$$
We can now state the partition of irreducible characters previously announced. 
 
\begin{prop}[\cite{HLM16}, Proposition 5.16] \label{prop:HLM}We have that 
$$\Irr(U)=\bigsqcup_{\Sigma \text{ antichain in } \Phi^+} \Irr(U)_{\Sigma}.$$
\end{prop}

\subsection{Cores and the Reduction algorithm} \label{sub:cores} In order to describe the sets of the form $\Irr(U)_{\Sigma}$ for 
an antichain $\Sigma$, we use the following procedure, explained in \cite[Section 3]{GLMP16}, and 
implemented in \cite{GAP4}. Our goal is to reduce from the study of some 
$\Irr(X_{\cS})_{\cZ}$, with $\cZ \subseteq \cZ(\cS)$, to the study of $\Irr(X_{\cS'})_{\cZ'}$, 
with $\cZ' \subseteq \cZ(\cS')$, such that $|\cS'| \lneq |\cS|$. 

We recall the following result.

\begin{prop}[\cite{GLMP16}, Lemma 3.1]\label{prop:small} Let $\cS:=\cP \setminus \cK$ be a quattern, let $\cZ \sub \cZ(\cS)$ and let $\gamma \in \cZ$.
Suppose that there exist $\delta, \beta \in \cS \setminus \{\gamma\}$, such that:
\begin{itemize}
\item[(i)] $\delta+\beta=\gamma$,

\item[(ii)] $\alpha + \alpha' \ne \beta$ for every $\alpha, \alpha' \in \cS$, and 

\item[(iii)] $\delta+\alpha \not\in \cS$ for every $\alpha \in \cS \setminus \{\beta\}$.

\end{itemize}
Let $\cP':=\cP \setminus \{\beta\}$ and $\cK':=\cK \cup \{\delta \}$.  Then we have that $\cS' := \cP' \setminus \cK'$ is a quattern with $X_{\cS'} \cong X_{\cP'}/X_{\cK'}$, and
the map 
\begin{align*}
\Irr(X_{\cS'})_\cZ &  \to \Irr(X_\cS)_\cZ \\
\chi & \mapsto \Ind^\beta \Inf_\delta \chi
\end{align*}
is a bijection of irreducible characters.
\end{prop}

The Reduction algorithm has been presented in \cite[Algorithm 3.3]{GLMP16} by applying repeatedly Proposition \ref{prop:small} to $\cS_{\Sigma}$. 
Here, we summarize it to recall tuples of the form $\fC=(\cS, \cZ, \cA, \cL, \cK)$ of positive roots, called \emph{cores}, and sets 
$\fO_1$, $\fO_2$ containing tuples of this form, such that we have a bijection 
$$\Irr(U)_{\Sigma} \longleftrightarrow \bigsqcup_{\fC \in \fO_1}\Irr(X_{\cS})_{\cZ} \sqcup \bigsqcup_{\fC \in \fO_2}\Irr(X_{\cS})_{\cZ}.$$
In particular, $X_{\cS}$ is abelian if and only if $\fC  \in \fO_1$, in which case we call $\fC$ an 
\emph{abelian core}; if $\fC  \in \fO_2$, we call $\fC$ a 
\emph{nonabelian core}. In the sequel we sometimes drop the whole notation $(\cS, \cZ, \cA, \cL, \cK)$ for 
$\fC$, and just refer to the pair $(\cS, \cZ)$ or to the quattern $\cS$. 

The reduction procedure is as follows. At each step of the procedure, we 
examine a tuple $(\cS, \cZ, \cA, \cL, \cK)$, where the set $\cS$ is a quattern with $\cZ \subseteq \cZ(\cS)$, the set 
$\cA$ (respectively $\cL)$ keeps a record of the roots of the form $\beta$ (respectively $\delta)$ 
at each step of the application of Proposition \ref{prop:small}, and the set $\cK$ keeps a record of the roots indexing 
root subgroups in the associated quattern group. The output of this procedure is 
the sets $\fO_1$ and $\fO_2$. We use in the sequel the notation $\Ind^\beta$, $\Inf_\delta$ and $\Ind^\cA$, $\Inf_\cK$ 
defined in \cite[\S2.3]{GLMP16}. 

\begin{itemize}
\item[\textbf{Setup.}] We initialize by putting $\cS=\cS_{\Sigma}$, $\cZ=\cZ(\cS_{\Sigma})$, and $\cA=\cL=\cK=\emptyset$ 
and $\fO_1=\fO_2=\emptyset$. 
\end{itemize}

Let us now assume $\fC=(\cS, \cZ, \cA, \cL, \cK)$ is constructed and taken into examination 
at this step of the procedure.

\begin{itemize}
\item[\textbf{Step 1.}] Let us assume that $\cS=\cZ(\cS)$. Then $X_{\cS}$ is abelian, and we can easily 
parametrize $\Irr(X_{\cS})_{\cZ}$. Namely, assume that $\cZ=\{\alpha_{i_1}, \dots, \alpha_{i_m}\}$ and 
$\cS \setminus \cZ=\{\alpha_{j_1}, \dots, \alpha_{j_n}\}$. Let us define 
$$\chi_{\ub}^{\ua}:=\Ind^{\cA} \Inf_{\cK} \lambda_{\ub}^{\ua},$$
where $\ub=(b_{j_1}, \dots, b_{j_n}) \in \F_q^n$, and $\ua=(a_{i_1}, \dots, a_{i_m}) \in (\F_q^{\times})^m$, 
and $\lambda_{\ub}^{\ua}$ is such that 
$$
\lambda_\ub^\ua(x_{\alpha_{i_k}}(t))=\phi(a_{i_k}t) \qquad  \text{ and } \qquad  \lambda_\ub^\ua(x_{\alpha_{j_h}}(t))=\phi(b_{j_h}t)
$$
for every $k=1, \dots, m$ and $h=1, \dots, n$. Then we have that 
$$\Irr(X_{\cS})_{\cZ}=\{\chi_{\ub}^{\ua} \mid \ub \in \F_q^n, \ua \in (\F_q^{\times})^m\}.$$
We add the element $\fC$ to the set $\fO_1$. 

\item[\textbf{Step 2.}] Let $\cS \ne \cZ(\cS)$, and let $\cR(\cS)$ be the set of pairs of the form $(\beta, \delta)$ 
satisfying the assumptions of Proposition \ref{prop:small}. Assume $\cR(\cS) \ne \emptyset$. We 
choose one particular element $(\beta, \delta) \in \cR(\cS)$, namely we choose $\delta$ to be maximal 
with respect to the linear ordering fixed on $\Phi^+$, and if $(\beta_1, \delta), \dots, (\beta_s, \delta)$ are in 
$\cR(\cS)$, we choose $\beta_i$ minimal with respect to the linear ordering on $\Phi^+$. Let us put 
$\fC':=(\cS', \cZ', \cA', \cL', \cK')$, with 
$$\cS'=\cS \setminus \{\beta, \delta\}, \quad \cZ'=\cZ, \quad \cA'=\cA \cup \{\beta\}, \quad \cL'=\cL \cup \{ \delta\}, \quad 
\cK'= \cK \cup \{\delta\}.$$
Then we have that
$$\Ind^\beta \Inf_\delta: \Irr(X_{\cS'})_\cZ \longrightarrow \Irr(X_\cS)_\cZ$$
is a bijection of irreducible characters. We continue by going back to Step 1 with $\fC'$ in place of $\fC$. 

\item[\textbf{Step 3.}] Let $\fC$ be such that $\cS \ne \cZ(\cS)$ and $\cR(\cS) = \emptyset$. Assume that 
$\cZ(\cS) \setminus (\cZ \cup \cD(\cS)) \ne \emptyset$, and let $\gamma$ be its maximal 
element with respect to the usual linear ordering. It is easy to see, 
as in \cite[Lemma 3.2]{GLMP16}, that 
$$\Irr(X_{\cS})_{\cZ}=\Irr(X_{\cS \setminus \{\gamma\}})_{\cZ} \sqcup \Irr(X_{\cS})_{\cZ \cup \{\gamma\}}.$$
Correspondingly, we continue by going back to Step 1 with each of the tuples 
$$\fC':=(\cS \setminus \{\gamma\}, \cZ, \cA, \cL, \cK \cup \{\gamma\}) \quad \text{and} \quad \fC'':=(\cS , \cZ\cup \{\gamma\}, \cA, \cL, \cK ).$$

\item[\textbf{Step 4.}] Let $\cS$ be such that $\cS \ne \cZ(\cS)$, $\cR(\cS) = \emptyset$ and 
$\cZ(\cS) \setminus (\cZ \cup \cD(\cS)) = \emptyset$. Then $X_{\cS}$ is not abelian 
and it cannot be reduced further using Proposition \ref{prop:small}. We add $\fC$ to $\fO_2$. 
The set $\Irr(X_{\cS})$ has to be investigated with different methods.

\end{itemize}

\smallskip

This algorithm has been implemented in \cite{GAP4} for all groups of rank $7$ or less. 
The numbers of nonabelian cores in each case are recorded in Table \ref{tab:corr}. 
The convention for the choice of $(\beta, \delta)$ as in Step 2 is slightly different from the one in \cite{GLMP16}, hence there are some different numbers of nonabelian cores for ranks $5$ or higher. 

We notice that if $\cD(\cS) \ne \emptyset$, then 
$$X_{\cS}=X_{\cS \setminus \cD(\cS)} \times X_{\cD(\cS)}, \,\,\,\, \text{hence} \,\,\,\, 
\Irr(X_{\cS})_{\cZ}=\Irr(X_{\cS \setminus \cD(\cS)})_{\cZ \setminus \cD(\cS)} \times 
\Irr(X_{\cD(\cS)})_{\cZ \cap \cD(\cS)}, $$
and $\Irr(X_{\cD(\cS)})_{\cZ \cap \cD(\cS)}$ is easily parametrized, as $X_{\cD(\cS)}$ is 
a direct product of its root subgroups. Then we assume in the sequel 
that we have a record of the set $\cD(\cS)$, 
and by slight abuse we 
identify $\cS$ with $\cS\setminus \cD(\cS)$ and $\cZ$ with $\cZ\setminus \cD(\cS)$.  

\begin{small}
\begin{center}
\begin{table}[t]
\begin{tabular}{|c|} 
\hline
$\rY_{r \le 3}$\\
\hline
0\\
\hline
\end{tabular}\quad\begin{tabular}{|c|c|c|c|}
\hline
$\rB_4$ & $\rC_4$ & $\rD_4$ & $\rF_4$\\
\hline
1 & 0 & 1 & 6\\
\hline
\end{tabular}\quad\begin{tabular}{|c|c|c|}
\hline
$\rB_5$ & $\rC_5$ & $\rD_5$\\
\hline
7 & 1 & 6\\
\hline
\end{tabular}\quad\begin{tabular}{|c|c|c|c|}
\hline
$\rB_6$ & $\rC_6$ & $\rD_6$ & $\rE_6$ \\
\hline
36 & 16 & 27 & 105 \\
\hline
\end{tabular}\quad\begin{tabular}{|c|c|c|c|}
\hline
$\rB_7$ & $\rC_7$ & $\rD_7$ & $\rE_7$ \\
\hline
245 & 129 & 160 & 3401 \\
\hline
\end{tabular}
\caption{The number of nonabelian cores in $U$, for $G$ of rank $7$ or less.}
\label{tab:corr}
\end{table}
\end{center}
\end{small}

\section{Isomorphism of nonabelian cores} \label{sec:nonc}

The behavior of the nonabelian cores is determined 
by the relations in the underlying quattern $\cS$. 
We want to record some invariants associated to 
such quatterns.

\begin{definition} \label{def:zmc} Let $\fC=(\cS, \cZ, \cA, \cL, \cK)$ be a nonabelian core. We say that $\fC$ is a \emph{$[z, m, c]$-core} if
\begin{itemize}
\item $|\cZ|=z$,
\item $|\cS| = m$, and
\item there are $c$ triples $(i, j, k)$, with $i < j$ and $\alpha_i, \alpha_j, \alpha_k \in \cS$, such that $\alpha_i+\alpha_j=\alpha_k$.
\end{itemize}
The triple $[z, m, c]$ associated to $\fC$ is called 
the \emph{form} of the nonabelian core $\fC$. 
\end{definition}

The focus of this section is to show that a triple $[z, m, c]$ uniquely determines the isomorphism 
type of a core $\fC$ in simply laced type when the rank of $G$ is $6$ or less. 

\begin{theorem} \label{theo:iso} Let $\Phi^+=\rD_6$ or $\Phi^+=\rE_6$, 
and let $\cS$, $\cS'$ be two quatterns of $\Phi^+$ corresponding 
to nonabelian cores of the form $[z, m, c]$. Then we have that $X_{\cS} \cong X_{\cS'}$.
\end{theorem}

The above theorem is proved computationally, following the procedure explained below. 
We start by recording in Table \ref{tab:coretype} the number of occurrences of each form 
of nonabelian cores in types $\rD_6$ and $\rE_6$. Notice that in these cases, for each 
$z$ and $m$ there exists a unique $c$ such that $[z, m, c]$ is a nonabelian core, then 
by Theorem \ref{theo:iso} the knowledge of $|\cS|$ and $|\cZ|$ tells apart the isomorphism 
type of $X_{\cS}$. 

\begin{table}[h]
\begin{center}
\begin{tabular}{|c|c|c|c|c||c|c|}
\cline{1-2} \cline{4-7} 
 \multicolumn{2}{|c|}{$\rD_6$} & \hspace{1.4cm} & \multicolumn{4}{|c|}{$\rE_6$} \\
\cline{1-2} \cline{4-7} 
 $[z, m, c]$ & $\#$ & \qquad & $[z, m, c]$ & $\#$ & $[z, m, c]$ & $\#$
 \\
\cline{1-2}  \cline{4-7} 
 $[3, 9, 6]$	 & $7$ 	 && $[3, 9, 6]$	 & $24$ & $[3, 10, 9]$	 & $45$ 	 \\
\cline{1-2} \cline{4-7} 
 $[3, 10, 9]$	 & $15$ 	 && $[4, 8, 4]$	 & $11$ & $[5, 10, 5]$	 & $1$ 	  \\
\cline{1-2} \cline{4-7} 
$[4, 18, 18]$	 & $1$ 	 && $[5, 12, 8]$	 & $2$ & $[5, 15, 11]$	 & $3$ 	  \\
\cline{1-2} \cline{4-7} 
$[4, 21, 28]$	 & $1$ 	&& $[5, 16, 15]$	 & $1$	&$[5, 20, 25]$	 & $1$   \\
\cline{1-2} \cline{4-7}
$[4, 24, 43]$	 & $1$ 	 && $[5, 21, 30]$	 & $1$  	&  $[6, 12, 6]$	 & $5$   \\
\cline{1-2} \cline{4-7}
$[5, 18, 18]$	 & $1$ 	 && $[6, 13, 7]$	 & $1$ 	& $[6, 14, 8]$	 & $3$   \\
\cline{1-2} \cline{4-7} 
$[6, 19, 20]$	 & $1$ 	 && $[6, 15, 12]$	 & $2$  	& $[6, 16, 12]$	 & $1$   \\
\cline{1-2} \cline{4-7} 
 \multicolumn{1}{c}{}		 & \multicolumn{1}{c}{}	  & \multicolumn{1}{c|}{}& $[6, 17, 17]$	 & $1$ 	& $[7, 15, 9]$	 &$3$   \\
 \cline{4-7}
\end{tabular}

\caption{The numbers of $[z, m, c]$-cores in types $\rD_6$ and $\rE_6$.}
\label{tab:coretype}
\end{center}
\end{table}

We recall the lower central series of $X_\cS$. For all $k\in\Z_{>0}$, the $k$-th member of the lower central series is denoted by $X_\cS^{(k)}$, recursively defined by $X_\cS^{(1)}:=X_\cS$ and $X_\cS^{(k+1)}:=[X_\cS,X_\cS^{(k)}]$. 
Notice that $X_\cS^{(k)}$ is always a quattern group, and $X_\cS^{(k+1)}<X_\cS^{(k)}$ whenever $X_\cS^{(k)}$ is nontrivial. 
We denote by $d$ the \emph{nilpotency class} of $X_\cS$, that is the unique 
positive integer such that $X_{\cS}^{(d)}\ne 1$ and $X_{\cS}^{(d+1)}= 1$.

The isomorphism test for cores is mainly based on the concept of local height, which is defined as follows.

\begin{definition}
\label{defn:loc_ht}
Let $\fC=(\cS, \cZ, \cA, \cL, \cK)$ be a core. 
A root $\alpha\in\cS$ is said to have \textit{local height} $k$ if $X_\alpha\subseteq X_\cS^{(k)}$ and $X_\alpha\not\subseteq X_\cS^{(k+1)}$.
\end{definition}

The roots in $\cS$ are then partitioned into their local height classes as $[\cS_1,\dots, \cS_d ]$, where $\cS_k$ denotes the set of all roots of local height $k$ in $\cS$ for every $k=1, \dots, d$. For any two cores $\fC$ and $\fC'$, it is clear that their quattern groups $X_{\cS}$ and $X_{\cS'}$ are not isomorphic if there exists a $k\ge 1$ such that $|\cS_{k}|\neq |\cS'_{k}|$. 

We say that a root $\alpha \in \cS_k$ is a \emph{lower bound} 
of a root $\delta \in \cS_{k+1}$ if there exists 
a root $\beta\in \cS_1$ such that $\alpha+\beta=\delta$. 
This naturally defines a lattice structure on $\cS$, whose 
suprema are the elements of $\cZ(\cS)$. Notice that 
if $\cS$ is a disconnected lattice with $\cS=\cS_a \sqcup 
\cS_b$, then $X_\cS=X_{\cS_a}\times X_{\cS_b}$; without 
loss of 
generality, we assume from now on that $\cS$ is 
connected. 

To show that the quattern groups of 
two cores 
of the form $[z, m, c]$ corresponding to $\cS$ and $\cS'$ 
are isomorphic, we proceed as follows.
\begin{itemize}
  \item[(a)] We find a lattice homomorphism between $\cS$ and $\cS'$. If it exists, we go to step (b).
  \item[(b)] Let $\rho$ be a lattice homomorphism between $\cS$ and $\cS'$. We try to lift $\rho$ to a group homomorphism $\varphi: X_{\cS} \to X_{\cS'}$ by checking the compatibility of the signs in the commutator relations between 
  root elements.
\end{itemize}

Let $\cS=[\cS_{1},\dots, \cS_d]$ and $\cS'=[\cS'_{1},\dots, \cS'_d]$ be two quatterns corresponding to a nonabelian core of the form $[z,m,c]$, with $|\cS_{k}|=|\cS'_{k}|$ for all $k=1, \dots, d$. For constructing a lattice homomorphism $\rho$ and lifting it up as a group isomorphism $\varphi$, we use the following algorithm. 

\textbf{Setup and base step.} 
For (a), we start with a setup of roots at the first local height layers of $\cS_{1}$ and $\cS_{1}'$, i.e. we 
choose a bijection $\rho$ from $\cS_{1}$ to $\cS_{1}'$. 

For (b), we set  $\varphi(x_\alpha(t)):=x_{\rho(\alpha)}(\pm t)$ for all $\alpha\in\cS_{1}$ and all $t\in\F_q$. This gives a setting for the first local height layer. Notice that the chosen sign `$+$' works in types $\rD_6$ and 
$\rE_6$, instead of trying every choice of 
the signs `$+$' and `$-$'. 

\textbf{Iterative step.} Assume that we constructed the $k$-th local height layer map, and $\cS_{k+1}$ and $\cS'_{k+1}$ are nonempty. We construct the maps $\rho$ for 
$(k+1)$-th local height layers, and $\varphi$ for root groups at $(k+1)$-th local height layers.

For (a), let $\delta\in\cS_{k+1}$, where $\rho(\delta)$ is yet to be defined. We find $\alpha\in\cS_{k}$ and $\beta\in\cS_{1}$ such that $\delta=\alpha+\beta$. If $\rho(\alpha)+\rho(\beta) \not\in \cS'_{k+1}$, then this construction ends here, and we return \emph{no solution} for the choice of $\rho$ from the first local height layer.
If $\rho(\alpha)+\rho(\beta) \in \cS'_{k+1}$, then we define $\rho(\delta):=\rho(\alpha)+\rho(\beta)\in\cS'_{k+1}$, and proceed further. 

For (b), we set $\varphi(x_\delta(t)):=x_{\rho(\delta)}(\epsilon_{\alpha,\beta} t)$ for every $t\in \F_q$, where $\epsilon_{\alpha,\beta}$ is determined as follows. 
If $\varphi(x_\alpha(t))=x_{\rho(\alpha)}(\epsilon_1 t)$, $\varphi(x_\beta(t))=x_{\rho(\beta)}(\epsilon_2 t)$
and $[x_\alpha(1),x_\beta(t)]=x_\delta(\epsilon_3 t)$ for all $t\in\F_q$, then  $\epsilon_{\alpha,\beta}:=\epsilon_1\epsilon_2 \epsilon_3$. Notice that 
in types $\rD_6$ and $\rE_6$ we have $\epsilon_1, \epsilon_2, \epsilon_3 \in\{\pm 1\}$. 

We check the compatibility of this setting  (that is, that the extension of $\varphi$ on $X_\delta$ is well-defined) by checking all other pairs $(\alpha',\beta') \in\cS_{k-i}\times\cS_{1+i}$ for $i=1, \dots, k-1$ such that $\alpha'+\beta'=\delta$. 
If the value $\epsilon_{\alpha,\beta}$ is unique, i.e. there is no pair $(\alpha',\beta')$ giving $\epsilon_{\alpha',\beta'} \ne \epsilon_{\alpha,\beta}$, then the mapping $x_\delta(t)\mapsto x_{\rho(\delta)}(\epsilon_{\alpha,\beta} t)$ is well-defined. Otherwise, this extension to $X_\delta$ is not well-defined, the construction ends here, and we return \emph{no solution} for this choice of $\rho$. 

Notice that when $p=2$, each choice for $\epsilon_{\alpha, \beta}$ is valid, thus the extension is always well-defined whenever we 
find an extension of $\rho$ such that $\rho(\delta)=\rho(\alpha)+\rho(\beta)\in\cS'_{k+1}$ as above.

If there is any other root in $\cS_{k+1}$, then we go back to the iterative step. 

\textbf{Output.} Assume that we constructed the $k$-local height layer map, and 
$\cS_{k+1}=\cS'_{k+1}=\emptyset$. 
We obtain the required group isomorphism $\varphi$ from $X_{\cS}$ onto $X_{\cS'}$.

\vspace{0.5cm} 
By using \cite{GAP4}, we apply the isomorphism test to types $\rD_6$ and $\rE_6$, and we check that every two nonabelian cores corresponding to the same triple $[z, m, c]$ are in fact isomorphic. 
Moreover, as previously remarked, in types $\rD_6$ and $\rE_6$ we only need the setup $\varphi(x_\alpha(t)):=x_{\rho(\alpha)}(t)$ for all $t\in \F_q$ in the 
base step, i.e. we do not have to check the negative sign choices.

However, Theorem \ref{theo:iso} does not generalize to higher rank, as we demonstrate with an example. 

\begin{example}\label{ex:E7} There exists 
a $[3, 9, 6]$-core $\fC$ in 
type $\rE_7$ such that 
\begin{itemize}
\item $\cS=\{\alpha_1,\alpha_5,\alpha_{14},\alpha_{17},\alpha_{20}, \alpha_{21}, \alpha_{22}, \alpha_{26}, \alpha_{37}\}$, 
\item $\cZ=\{\alpha_{21}, \alpha_{26},  \alpha_{37}\}$, 
\item $\cA = \{\alpha_2, \alpha_3, \alpha_6, \alpha_{7}, \alpha_8, \alpha_{10}, \alpha_{12}, \alpha_{13}, \alpha_{15}, \alpha_{24}, \alpha_{29}, \alpha_{31}, \alpha_{35}, \alpha_{36}\}$ and
$\cL = \{\alpha_{18}, \alpha_{19}, \alpha_{23},$  $\alpha_{25}, \alpha_{27}, \alpha_{28}, \alpha_{30}, \alpha_{34}, \alpha_{39}, \alpha_{40}, \alpha_{41}, \alpha_{42}, \alpha_{44}, \alpha_{45}\}$.
\end{itemize}
 Using the methods described in $\S \ref{sub:3dim}$, we obtain that $\Irr(X_{\cS })_{\cZ(\cS)}$ consists of $q^2(q-1)^3$ irreducible characters of degree $q^2$ for every prime $p$. On the other hand, we find in $\rE_7$ another nonabelian 
$[3, 9, 6]$-core whose underlying quattern group $X_{\cS'}$ is conjugate to the 
quattern group arising from the unique nonabelian core lying inside the 
natural quotient $\rU\rD_4(q)$ of $\rU\rE_7(q)$; 
this has been studied in \cite[\S4]{HLM11}. As recorded in Table \ref{tab:coresD6}, we have that $\Irr(X_{\cS' })_{\cZ(\cS')}$ consists instead of $(q-1)^3$ irreducible characters of degree $q^3$ if $p \ge 3$. 

On the one hand, under the action of the Weyl group of $\rE_7$ on the root system we get that $X_{\cS}$ is isomorphic to the quattern group corresponding to  $\Phi^+\setminus\{\alpha_1,\alpha_3,\alpha_5=\alpha_1+\alpha_3\}$ in the root system of type $\rD_4$. On the other hand, it is noted that the 
$[3, 9, 6]$-cores we obtain in types $\rD_6$ and $\rE_6$ are conjugate to the quattern $\{\alpha_1,\dots,\alpha_{10}\} \setminus \{\alpha_3\}$ in type 
$\rD_4$, which does 
correspond to 
the only nonabelian core in type $\rD_4$. 
\end{example}

\section{A reduction process for nonabelian cores}\label{sec:cors}

\subsection{A reduction lemma.} \label{sub:3dim}

A method for the study of nonabelian cores 
is presented in \cite[\S4.2]{GLMP16}. In this subsection 
we present a slight variation of the setup and the 
method, in order to have a direct 
algorithmic application to the study of 
the corresponding quatterns. 

Throughout the rest of this subsection, we assume that $V$ is a finite group, 
$H$ is a subgroup of $V$ with fixed transversal $X$ in $V$, and $Y, Z$ are subgroups of $V$, such that 
\begin{itemize}
\item[(i)] $Z \subseteq Z(V)$, 
\item[(ii)] $Y \subseteq Z(H)$,
\item[(iii)] $Z \cap Y=1$, and
\item[(iv)] $[X, Y] \subseteq Z$. 
\end{itemize}
Moreover, we fix $\lambda \in \Irr(Z)$, and we define 
$$X':=\{x \in X \mid \lambda([x, y])=1 \text{ for all } y \in Y\}, 
\quad Y':=\{y \in Y \mid \lambda([x, y])=1 \text{ for all } x \in X\},$$
and $\hlm:=\Inf_Z^{YZ} \lm$. As $[Y, H]=1$, we have $[Y, V]=[Y, HX]=[Y, X] \subseteq Z$, hence $YZ \teq V$ and $V$ acts on $\Irr(YZ)$.

\begin{lemma} \label{lem:stab}
We have $I_V(\hlm) =HX'$.
\end{lemma} 

\begin{proof}
Let $h \in H$ and $x \in X$. For every $y \in Y$ and $z \in Z$, we have that 
$${}^{hx}\hlm(yz)=\hlm(y^{hx}z)=\hlm(y^{x}z)=\hlm(y[y, x]z)=\hlm(yz)\lm([y, x]).$$
Then ${}^{hx}\hlm=\hlm$ if and only if $\lambda([x, y])=1$ for every $y \in Y$, that is $x \in X'$. 
\end{proof}

Let us define $H':=HX'$, and let us fix a transversal $\tilde X$ of 
$H'$ in $V$. For every $\tilde x \in \tilde X$, let us put $\psi_{\tilde x}:=({}^{\tilde x}\hlm)|_Y \in \Irr(Y)$. 
It is easy to check that $\psi_{\tilde x}(y)=\hlm ( [y, \tilde{x}])$ 
for every $y \in Y$. Let us put 
$W_{\tilde{X}}:=\{\psi_{\tilde x} \mid \tilde x \in \tilde X\}$. 

\begin{lemma}\label{lem:dual}
 $W_{\tilde{X}}$ is a subgroup of $\Irr(Y)$, and 
$$|\tilde{X}|=|W_{\tilde{X}}|=|Y:Y'|.$$
\end{lemma}

\begin{proof}
From the fact that $[Y, V]\subseteq Z$ and  $\lambda([h', y])=1$ for every $h' \in H'$, it easily follows that $W_{\tX}$ is 
a subgroup of $\Irr(Y)$, and that $\psi_{\tx_1} \ne \psi_{\tx_2}$ if 
$\tx_1 \ne \tx_2$, hence also $|\tX|=|W_{\tX}|$.  
Finally, notice that 
$$Y'=\bigcap_{\tx \in \tX} \psi_{\tx}=\bigcap_{\eta \in W_{ \tX}} \ker(\eta),$$
hence $|W_{\tX}|=|Y:Y'|$.
\end{proof}

From now on we need an extra assumption on $Y'$, 
namely 
\begin{itemize}
\item[(v)] $Y'$ has a complement $\tY$ in $Y$.
\end{itemize}

\begin{prop} \label{prop:newrl}
We have a bijection 
\begin{equation}\label{eq:bijn1}
\Ind_{H'}^{V} \Inf_{H'/\tY}^{H'}: \Irr(H'/\tY \mid \lambda) \longrightarrow \Irr(V \mid \lambda).
\end{equation}
\end{prop}

\begin{proof} 
Let $\tlm:=\Inf_{Z}^{Z\tY}\lm$. 
Lemma \ref{lem:stab} 
yields 
$I_V(\tlm)=HX'$ and ${}^{\tx_1}\tlm \ne {}^{\tx_1}\tlm$ if 
$\tx_1 \ne \tx_2$. Hence
\begin{equation}\label{eq:sup}
\Irr(Z\tY \mid \lm)=\{ {}^{\tx}\tlm \mid \tx \in \tX\}.
\end{equation} 
Since $Z\tY \teq V$, by Clifford's theory we have a bijection
$$\Ind_{H'}^{V}: \Irr(H' \mid \tlm) \longrightarrow \Irr(V \mid \tlm).$$
By identifying $\Irr(H' \mid \tlm)$ with $\{\eta \in \Irr(H') \mid \tY \subseteq \ker(\eta) \}$, the above yields the bijection
$$\Ind_{H'}^{V} \Inf_{H'/\tY}^{H'}: \Irr(H'/\tY \mid \lm) \longrightarrow \Irr(V \mid \tlm).$$
We have that $\Irr(V \mid \tlm)=\Irr(V \mid \lm) \cap \Irr(V \mid 1_{\tY}) \subseteq \Irr(V \mid \lm)$. 

The claim is then proved if we show $\Irr(V \mid \lm) \subseteq \Irr(V \mid \tlm)$. If 
$\chi \in \Irr(V \mid \lm)$, we have $\langle \chi|_{Z\tY}, \lm_{Z\tY}\rangle \ne 0$. Let then $\eta \in \Irr(Z\tY \mid \lm)$ such 
that $\chi \in \Irr(V \mid \eta)$. Then 
$$\tlm \in \{{}^{\tx}\tlm \mid \tx \in \tX\}
= \{{}^{\tx}\eta \mid \tx \in \tX\}
\subseteq \{{}^{g}\eta \mid g \in V  \}
=\{\mu \in \Irr(Z\tY) \mid \langle \chi|_{Z\tY}, \mu\rangle \ne 0\},$$
where the first equality holds by Equation 
\eqref{eq:sup}, and the second equality holds by Clifford's theory. Hence $0 \ne 
\langle \chi|_{Z\tY}, \tlm\rangle$, that is, $\chi \in \Irr(V \mid \tlm)$. \end{proof}

\begin{definition}
The $X$ and $Y$ defined at the start of this section 
are called \emph{candidate for an arm} and \emph{candidate for a leg} respectively, and $\tilde{X}$ and $\tilde{Y}$ are called \emph{arm} and \emph{leg} respectively. 
\end{definition}

The terminology of arms and legs is 
motivated by the case $U=\rU\rA_r(q)$, as remarked 
in \cite[Section 6]{HLM16}.

Let $\overline{V}:=H'/(\tY\ker\lm)$. We observe 
that $Y' \subseteq Z(\overline{V})$. Before stating a consequence of 
Proposition \ref{prop:newrl}, we introduce some notation that is 
frequently used in the sequel. 

\begin{definition}\label{def:genroot} Let $\cS$ be a quattern, and let $i_1, \dots, i_m$ be such that the subset 
$\{\alpha_{i_1}, \dots, \alpha_{i_m}\}$ of $\cS$ satisfies $\cZ(\{\alpha_{i_1}, \dots, \alpha_{i_m}\})=\{\alpha_{i_1}, \dots, \alpha_{i_m}\}$. 
For fixed $c_1, \dots, c_m \in \F_q^\times$, we define 
$$x_{i_1, \dots, i_m}^{c_1, \dots, c_m}(t):=x_{i_1}(c_1t)\cdots x_{i_m}(c_mt)\ \ \text{ for all }t\in\F_q.$$
Moreover, we put 
$$X_{i_1, \dots, i_m}^{c_1, \dots, c_m}:=\{x_{i_1, \dots, i_m}^{c_1, \dots, c_m}(t) \mid t \in \F_q\}.$$
\end{definition}

We usually drop the explicit labels $c_1, \dots, c_m \in \F_q^\times$ and we write 
$x_{i_1, \dots, i_m}(t)$ instead of $x_{i_1, \dots, i_m}^{c_1, \dots, c_m}(t)$, and $X_{i_1, \dots, i_m}$ 
instead of $X_{i_1, \dots, i_m}^{c_1, \dots, c_m}$; the choice of $c_1, \dots, c_m$ will be made explicit 
when needed. Notice that 
as $\cZ(\{\alpha_{i_1}, \dots, \alpha_{i_m}\})=\{\alpha_{i_1}, \dots, \alpha_{i_m}\}$, we have 
that $X_{i_1, \dots, i_m} \cong (\F_q, +)$. Moreover, if $\ui:=(i_1, \dots, i_m)$ with $1 \le i_1 < \dots < i_m \le |\cS|$, 
we denote by $x_{\ui}(t)$ the element $x_{i_1, \dots, i_m}(t)$, similarly for $X_{\ui}$. 

\begin{cor} \label{cor:cornewrl}
Assume that $Y'=X_{\ui_1} \times \cdots \times X_{\ui_s}$, where 
$\ui_1, \dots, \ui_s$ are lexicographically ordered. For every $\ub=(b_1, \dots, b_s) \in \F_q^s$, let 
$$
K_{\ub}:=\prod_{j=1, \dots, s \mid  b_j=0} X_{\ui_j} \quad \text{and} \quad V_{\ub}:=\overline{V}/K_{\ub},$$
and let $\mu_{\ub} \in \Irr(Y'/K_{\ub})$ be such that 
$\mu_{\ub}(x_{\ui_j}(t))=\phi(b_jt)$ for every $j=1, \dots, s$. 
Then 
$$\Irr(V \mid \lambda)\cong\bigsqcup_{\ub \in \F_q^s} \Irr(V_{\ub} \mid \lambda \otimes \mu_{\ub}).$$
\end{cor}
For example, let us suppose that $Y'$ is 
a diagonal subgroup of $X_{\cJ}$ isomorphic to $\F_q$, that is $s=1$. Then
$$\Irr(V \mid \lambda)\cong\Irr(\overline{V}/Y' \mid  \lambda) \sqcup  \bigsqcup_{\mu \in \Irr(Y')\setminus \{1_{Y'}\}}\Irr(\overline{V} \mid \lambda \otimes \mu).$$

We are interested in applying 
Proposition \ref{prop:newrl} and 
Corollary \ref{cor:cornewrl} to the setting of quattern 
groups. Given a quattern in rank $6$ or less, the 
validity of the assumptions of the following result is easy to check by 
using \cite{GAP4}. 

\begin{cor}\label{cor:plus}
Let $\cS=\cP \setminus \cK$ be a quattern. 
Assume that there exist subsets 
$\cZ$, $\cI$ and $\cJ$ of $\cS$, such that 
\begin{itemize}
\item[(0)] $\cS \setminus \cI$ is a quattern, 
\item[(i)] $\cZ\subseteq \cZ(\cS)$, 
\item[(ii)] $\cJ\subseteq \cZ(\cS \setminus \cI)$, 
\item[(iii)] $\cJ \cap \cZ = \emptyset$, and 
\item[(iv)] $\alpha \in \cI, \beta \in \cJ, \alpha + \beta \in \cS$ 
$\Rightarrow$ $\alpha+\beta \in \cZ$. 
\end{itemize}
Let us put $Z=X_\cZ$, $X=X_\cI$, $Y=X_\cJ$ and $H=X_{\cS \setminus \cI}$, 
and define $X'$, $Y'$ and $H'$ as in Proposition \ref{prop:newrl} and $\overline{V}$ as in Corollary \ref{cor:cornewrl}. Then we have a bijection 
$$\Ind_{\tilde{H}}^G \Inf_{\overline{V}}^{H'}: \Irr(\overline{V} \mid \lambda) \longrightarrow \Irr(X_{\cS} \mid \lambda).$$
\end{cor}

\begin{proof}
By $\S$\ref{ss:quattern}, it is easy to check that (0) is equivalent 
to $H$ being a subgroup of $X_{\cS}$ and each of (i) to (iv) is 
equivalent to the corresponding assumption in Proposition \ref{prop:newrl}. Moreover, (v) is clear as 
$Y$ is elementary abelian. 
\end{proof}

\begin{remark} \label{rem:zuij}
Let us assume that $\cS=\cZ \cup \cI \cup \cJ$, such that assumption (ii) of 
Corollary \ref{cor:plus} is satisfied. Then $\cS \setminus \cI$ is 
automatically a quattern. 
\end{remark}

\subsection{Graphs of nonabelian cores} \label{sub:graph}
Let us fix a nonabelian core $\fC$ corresponding to 
$\cS$ and $\cZ$. In order to check the assumptions 
of Corollary \ref{cor:plus}, 
we define a graph associated to $\fC$. 

\begin{definition} \label{def:graph}
Let $\fC$ be a nonabelian core corresponding 
to $\cS$ and $\cZ$. 
We say that 
$\alpha, \beta \in \cS$ are \emph{$\cZ$-connected}, or just \emph{connected}, if $\alpha+\beta=\gamma$ 
with $\gamma \in \cZ$. 
\end{definition}

With this definition, we regard $\fC$ as a graph whose vertices are the 
elements of $\cS \setminus \cZ$, and there is an edge between 
$\alpha$ and $\beta$ if and only if $\alpha$ and 
$\beta$ are $\cZ$-connected. We then have the usual notion of 
\emph{connected components}. We say that the \emph{heart} $\cH$ of the core $\fC$ is 
the set of roots $\alpha \in \cS \setminus \cZ$ such that $\{ \alpha \}$ is a connected 
component on its own. If $\cH = \emptyset$, we say that the underlying core 
$\fC$ is a \emph{heartless core}. Otherwise, we call $\fC$ a \emph{core with a heart}. 

We now define some important cycles in $\fC$, whose analysis 
allows us to have a systematic procedure to reduce to 
the study of irreducible characters of smaller subquotients of $X_{\cS}$. 

\begin{definition}\label{def:circle}
We say that 
$\cC:=\{\beta_1, \dots, \beta_s\}\subseteq \cS$ with $\beta_1, \dots, \beta_s$ distinct is a 
\emph{circle} in $\cS$ if 
$\beta_i$ is connected 
to 
$\beta_{i+1}$ for $i=1, \dots, s-1$, and $\beta_s$ is connected 
to $\beta_1$.
\end{definition}

\begin{figure}[h]
\begin{center}
\begin{minipage}{0.5\textwidth}

\begin{center}
\begin{tikzpicture}[scale=0.75, crc/.style={circle,draw=black!50,thick,
inner sep=0pt,minimum size=8mm}, td/.style={rectangle,draw=black,thick, dotted,
inner sep=0pt,minimum size=8mm},  rrt/.style={circle,draw=red!50,thick,
inner sep=0pt,minimum size=8mm}, brt/.style={rectangle,draw=blue!50,thick,
inner sep=0pt,minimum size=8mm}, transform shape]
 \node (a) at (0:2.25)   {$\alpha_2$};
 \node (b) at (36:2.25)   {$\alpha_{12}$};
  \node (c) at (72:2.25)   {$\alpha_{5}$};
 \node (d) at (108:2.25)     {$\alpha_9$};
 \node (e) at (144:2.25)    {$\alpha_{11}$};
 \node (f) at (180:2.25)  {$\alpha_8$};
 \node (g) at (216:2.25)   {$\alpha_7$};
  \node (h) at (252:2.25)    {$\alpha_{10}$};
 \node (i) at (288:2.25)     {$\alpha_3$};
 \node (j) at (324:2.25)   {$\alpha_{16}$};
 \node (z) at (0, 0)   {$\alpha_4$};
\draw (a) -- (b) -- (c) -- (d) -- (e) -- (f) -- (g) -- (h) -- (i) -- (j) -- (a) ;
\end{tikzpicture}
\end{center}
\end{minipage}\vline\begin{minipage}{0.5\textwidth}
\begin{center}
\begin{tikzpicture}[scale=0.75, crc/.style={circle,draw=black!50,thick,
inner sep=0pt,minimum size=8mm}, td/.style={rectangle,draw=black,thick, dotted,
inner sep=0pt,minimum size=8mm},  rrt/.style={circle,draw=red!50,thick,
inner sep=0pt,minimum size=8mm}, brt/.style={rectangle,draw=blue!50,thick,
inner sep=0pt,minimum size=8mm}, transform shape]
 \node (a) at (3/2, 1.5/2)   {$\alpha_{15}$};
 \node (b) at (1.5/2, 4.5/2)  {$\alpha_2$};
  \node (c) at (-1.5/2, 4.5/2)  {$\alpha_9$};
 \node (d) at (-3/2, 1.5/2)     {$\alpha_7$};
 \node (e) at (-3/2, -1.5/2)   {$\alpha_{20}$};
  \node (f) at (-1.5/2, -4.5/2) {$\alpha_4$};
 \node (g) at (1.5/2, -4.5/2)   {$\alpha_8$};
  \node (h) at (3/2, -1.5/2)    {$\alpha_{10}$};
 \node (i) at (-5/2, 1.5/2)     {$\alpha_{16}$};
 \node (j) at (-5/2, -1.5/2)    {$\alpha_{12}$};
 \node (k) at (5/2, -1.5/2)    {$\alpha_{11}$};
 \node (l) at (5/2, 1.5/2)    {$\alpha_{14}$};
  \node (z) at (0, 0)   {$\alpha_3$};
\draw (a) -- (b) -- (c) -- (d) -- (e) -- (f) -- (g) -- (h) -- (a);
\draw (c) -- (i) -- (j) -- (f);
\draw (g) -- (k) -- (l) -- (b);
\end{tikzpicture}
\end{center}
\end{minipage}
\end{center}

\caption{The graphs of the core of the form $[5, 16, 15]$ in $\rD_6$ and of 
the core of the form $[6, 19, 20]$ in $\rE_6$.}
\label{fig:exgra}
\end{figure}
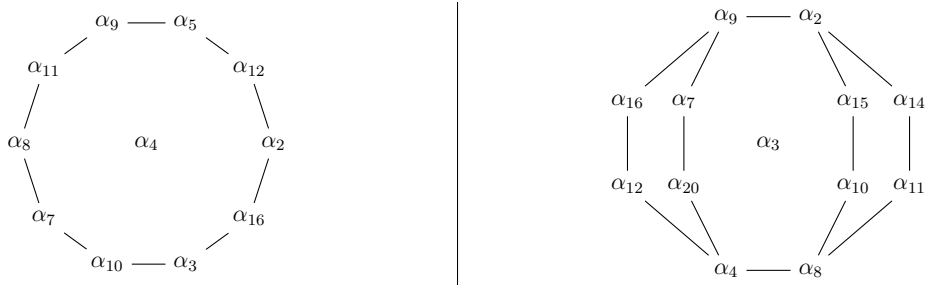

\vspace{-2mm}

The goal of the rest of this section is to construct \emph{unique} 
$\cI$ and $\cJ$ satisfying the assumptions 
of Corollary \ref{cor:plus} for each nonabelian core $\fC$. 

\begin{remark} \label{rem:1circ}
Let us assume that $\cS$ contains just a \emph{single} circle $\cC$. 
Let 
$\cC=\{\beta_1 \dots, \beta_s\}$ in the notation of Definition \ref{def:circle}. 
We start by describing a construction of such $\cI$ and $\cJ$ in this particular case. 
Let us define $\beta_0:=\beta_s$ and $\beta_{s+1}:=\beta_1$. 
We first define $\delta_1$ to be the minimal root in $\cC$ with 
respect to the usual linear ordering on roots. If $\delta_1=\beta_{j_1}$, then we choose 
$\delta_2$ to be the maximum of $\beta_{j_1-1}$ and $\beta_{j_1+1}$. 

Now assume that 
$\delta_i$ is constructed for $2 \le i \le s-1$. Then $\delta_i=\beta_{j_i}$ 
for some $\beta_{j_i}$, hence $\delta_{i-1} \in \{\beta_{j_i-1}, 
\beta_{j_i+1}\}$. 
If $\delta_{i-1}=\beta_{j_i-1}$, we define $\delta_{i+1}:=\beta_{j_i+1}$. Viceversa 
if $\delta_{i-1}=\beta_{j_i+1}$, we define $\delta_{i+1}:=\beta_{j_i-1}$.
Notice that $\delta_{s}$ is connected to 
$\delta_1$. If $s=2m$ is even, then we put 
$\cI:=\{\delta_1, \delta_3, \ldots, \delta_{2m-1}\}$ and 
$\cJ:=\{\delta_2, \delta_4, \ldots, \delta_{2m}\}$. 
If $s=2m+1$ is odd, then we put 
$\cI:=\{\delta_1, \delta_3, \ldots, \delta_{2m-1}, \delta_{2m+1}\}$ and 
$\cJ:=\{\delta_2, \delta_4, \ldots, \delta_{2m}\}$. 

We now check the conditions of Corollary \ref{cor:plus}. Assumptions (i), (iii) and (iv) clearly hold. 
If two roots in $\cJ$ were connected to each other, 
we would have a smaller circle $\cC'$ in $\cC$, which contradicts the assumption of $\cC$ being 
the unique circle in $\cS$. Hence $\cZ(\cJ)=\cJ$, which implies $\cJ 
\subseteq \cZ(\cS \setminus \cI)$. Then (ii) is satisfied, and by Remark \ref{rem:zuij} 
we have that (0) is also satisfied. Therefore, 
$\cI$ and $\cJ$  
satisfy all the assumptions of the corollary.
\end{remark}

\begin{figure}[h] 
\begin{center}
\begin{minipage}{0.5\textwidth}
\begin{center}
\begin{tikzpicture}[scale=0.75, crc/.style={circle,draw=black!50,thick,
inner sep=0pt,minimum size=8mm}, td/.style={rectangle,draw=black,thick, dotted,
inner sep=0pt,minimum size=8mm},  rrt/.style={circle,draw=red!50,thick,
inner sep=0pt,minimum size=8mm}, brt/.style={rectangle,draw=black,thick,
inner sep=0pt,minimum size=8mm}, transform shape]
 \node (a) at (0:2) [td]  {$\alpha_2$};
 \node (b) at (60:2) [brt]  {$\alpha_{14}$};
  \node (c) at (120:2)   [td] {$\alpha_{11}$};
 \node (d) at (180:2) [brt]    {$\alpha_8$};
 \node (e) at (240:2)  [td]  {$\alpha_{10}$};
  \node (f) at (300:2) [brt]   {$\alpha_5$};
\draw (a) -- (b) -- (c) -- (d) -- (e) -- (f) -- (a) ;
\end{tikzpicture}
\end{center}
\end{minipage}\vline\begin{minipage}{0.5\textwidth}
\begin{center}
\begin{tikzpicture}[scale=0.75, crc/.style={circle,draw=black!50,thick,
inner sep=0pt,minimum size=8mm}, td/.style={rectangle,draw=black,thick, dotted,
inner sep=0pt,minimum size=8mm},  rrt/.style={circle,draw=red!50,thick,
inner sep=0pt,minimum size=8mm}, brt/.style={rectangle,draw=black,thick,
inner sep=0pt,minimum size=8mm}, transform shape]
 \node (a) at (0:2) [td]  {$\alpha_1$};
 \node (b) at (72:2) [brt]  {$\alpha_3$};
  \node (c) at (144:2)   [td] {$\alpha_4$};
 \node (d) at (216:2) [brt]    {$\alpha_5$};
 \node (e) at (288:2)  [td]  {$\alpha_{13}$};
\draw (a) -- (b) -- (c) -- (d) -- (e) -- (a) ;
\end{tikzpicture}
\end{center}
\end{minipage}
\end{center}

\caption{The $\cI$ and $\cJ$ constructed in Remark \ref{rem:1circ}, corresponding to roots in a dotted box and roots in a straight box 
respectively, in the cases of the circles corresponding to the $[3, 9, 6]$-core of $\rD_6$ and 
to the $[5, 10, 5]$-core of $\rE_6$.}
\label{fig:covcir}
\end{figure}
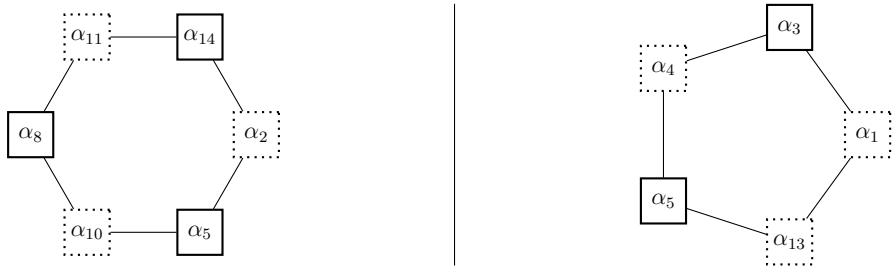

\vspace{-2mm}

We now construct $\cI$ and $\cJ$ for types 
$\rD_6$ and $\rE_6$ in the case when two or more circles occur in $\cS$. 
The number of circles of each nonabelian core in rank $6$ or less is relatively small, 
and we check by \cite{GAP4} that the $\cI$ and $\cJ$ obtained as follows satisfy 
the assumptions of Corollary \ref{cor:plus} for all nonabelian cores, except the $[4, 24, 43]$-core in type $\rD_6$, which 
is examined in full details in \S\ref{sub:chea}. 
For several nonabelian cores 
in type $\rE_7$, the $\cI$ and $\cJ$ constructed in this way do not satisfy the 
conditions of Corollary \ref{cor:plus}. One of the aims of subsequent 
work is to refine the following construction for higher numbers of circles in a nonabelian core. 

\textbf{Setup.} We collect all 
the distinct circles $\cC_1, \dots, \cC_t$ in $\cS$, ordered such that 
\begin{itemize}
\item if $|\cC_i|$ is even and $|\cC_j|$ is odd, then $i<j$;
\item if $|\cC_i|\ne |\cC_j|$ have the same parity, then 
$i<j$ if and only if $|\cC_i|> |\cC_j|$; and 
\item if $|\cC_i|= |\cC_j|$, then 
$i<j$ if and only if 
$\min\{k \mid \alpha_k \in \cC_i \setminus \cC_j\}<
\min\{k \mid \alpha_k \in \cC_j \setminus \cC_i\}$.
\end{itemize}
As circles are determined just by the relations among roots in $\cS$, the 
sets $\cC_1$, \dots, $\cC_t$ can easily be determined by using \cite{GAP4}. 
We decide to look first at the circles with even cardinality 
because the construction of arms and legs of smaller even circles is often 
compatible with the one of bigger even circles, while 
in general we have less compatibility among odd circles.  

\textbf{Base step.} We start by looking at $\cC_1$. 
We define $\cI_1$ and $\cJ_1$ 
as we determined 
$\cI$ and $\cJ$ in Remark \ref{rem:1circ}, 
namely we decide the 
minimum root $\delta$ of $\cC_1$ to be in $\cI_1$, and we alternate adjoining the remaining roots of $\cC_1$ into $\cJ_1$ and $\cI_1$ in the direction 
of the maximum neighbor of $\delta$ in $\cC_1$. 

\textbf{Iterative step.} Let us suppose that 
$\cI_k$ and $\cJ_k$ have been constructed with $k<t$. 
Then there exists another circle $\cC_{k+1}$ after 
$\cC_k$ in the ordering previously fixed. 
We now define two sets $\cI(\cC_{k+1})$ and 
$\cJ(\cC_{k+1})$ such that $\cI(\cC_{k+1}) \cup \cJ(\cC_{k+1})
=\cC_{k+1}$, and then we take advantage of them to construct $\cI_{k+1}$ and $\cJ_{k+1}$. 

If $\cC_{k+1}\cap \cI_{k} \ne \emptyset$, 
then 
we construct 
$\cI(\cC_{k+1})$ and $\cJ(\cC_{k+1})$ 
exactly as we constructed 
$\cI$ and $\cJ$ in Remark \ref{rem:1circ} 
respectively. If 
$\cC_{k+1}\cap \cI_{k} = \emptyset$ and 
$\cC_{k+1}\cap \cJ_{k} \ne \emptyset$, then we 
let $\cC_{k+1}=\{\beta_{1} \dots, \beta_{s}\}$ be as in Definition \ref{def:circle}, with $\beta_{0}:=\beta_{s}$ 
and $\beta_{s+1}:=\beta_{1}$. 
We let $\delta_{1}$ be the maximum root in $\cC_{k+1}\cap \cJ_{k}$
; we have $\delta_{1}=\beta_{ j_1}$ 
for some $j_1 \in \{1, \dots, s\}$. Then we let 
$\delta_{ 2}$ be the minimum of the roots $\beta_{ j_1-1}$ and $\beta_{ j_1+1}$. 
If $2 \le i \le s-1$ and $\delta_{i}$ is constructed such that 
$\delta_{i}=\beta_{ j_i}$, we define $\delta_{i+1}$ to be 
$\beta_{ j_i+1}$ in the case $\delta_{ i-1}=\beta_{j_i-1}$, 
and $\beta_{j_i-1}$ in the case $\delta_{ i-1}=\beta_{ j_i+1}$. 
If $s=2m$ then we put $\cJ(\cC_{k+1}):=\{\delta_{1}, \delta_{3}, \ldots, 
\delta_{2m-1}\}$
and 
$\cI(\cC_{k+1}):=\{\delta_{2}, \delta_{4}, \ldots, 
\delta_{2m}\}$; otherwise $|\cC|=2m+1$ and we put 
$\cJ(\cC_{k+1}):=\{\delta_{1}, \delta_{3}, \ldots, 
\delta_{2m-1}\}$ and $\cI(\cC_{k+1}):=\{\delta_{2m+1}\}\cup \{\delta_{2}, \delta_{4}, \ldots, 
\delta_{2m}\}$. If 
$\cC_{k+1}\cap \cI_{ k} = 
\cC_{k+1}\cap \cJ_{ k} = \emptyset$, 
again we proceed in the same way as Remark \ref{rem:1circ} to construct $\cI(\cC_{k+1})$ 
and $\cJ(\cC_{k+1})$. Finally, we define $\cI_{k+1}:=\cI_k \cup \cI(\cC_{k+1})$ and 
$\cJ_{k+1}:=(\cJ_k \cup \cJ(\cC_{k+1})) \setminus \cI_{k+1}$. 

\textbf{Output.} The sets $\cI_t$ and $\cJ_t$ are constructed. 
We define $\cI:=\cI_t$ and $\cJ:=\cJ_t$. These sets turn out 
to satisfy our desired properties. 

\begin{lem} \label{lem:corcov} Let $\fC$ be a nonabelian core of $\rD_6$ not of the form $[4, 24, 43]$ 
or a nonabelian core of $\rE_6$, corresponding to $\cS$ 
and $\cZ$. Then the sets $\cI$ and $\cJ$ constructed as 
above satisfy the assumptions of Corollary \ref{cor:plus}.
\end{lem}

\begin{proof}
This is a straightforward check in \cite{GAP4}. 
\end{proof}

An explicit form of the sets $\cI$ and $\cJ$ 
for each nonabelian core in types $\rD_6$ and $\rE_6$ is 
provided in \cite{LMP}. Through the construction pointed out above, we can 
reduce to smaller subquotients of $X_{\cS}$. The analysis of $\Irr(X_{\cS})$ for a heartless core is straightforward, 
once the sets $X'$ and $Y'$ are 
explicitly known, as such subquotients turn out to be abelian, except in the case of the 
$[6, 16, 12]$-core of $\rE_6$; 
the study of heartless cores is explained in $\S$\ref{sub:heartless}. 
The analysis of nonabelian cores with a heart, detailed in 
$\S$\ref{sub:chea}, is more complicated. In particular, the case of the $[4, 24, 43]$-core of 
$\rD_6$ is not covered by Lemma \ref{lem:corcov}; we treat it by applying directly Proposition \ref{prop:newrl}.

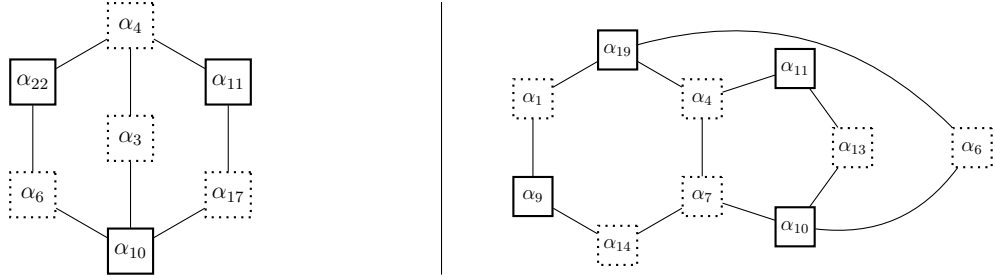
\begin{figure}[h]
\begin{center}
\begin{minipage}{0.5\textwidth}
\begin{center}
\begin{tikzpicture}[scale=0.75, crc/.style={circle,draw=black!50,thick,
inner sep=0pt,minimum size=8mm}, td/.style={rectangle,draw=black,thick, dotted,
inner sep=0pt,minimum size=8mm},  rrt/.style={circle,draw=red!50,thick,
inner sep=0pt,minimum size=8mm}, brt/.style={rectangle,draw=black,thick,
inner sep=0pt,minimum size=8mm}, transform shape]
 \node (a) at (30:2) [brt]  {$\alpha_{11}$};
 \node (b) at (90:2) [td]  {$\alpha_{4}$};
  \node (c) at (150:2)   [brt] {$\alpha_{22}$};
 \node (d) at (210:2) [td]    {$\alpha_6$};
 \node (e) at (270:2)  [brt]  {$\alpha_{10}$};
  \node (f) at (330:2) [td]   {$\alpha_{17}$};
  \node (g) at (0, 0) [td]   {$\alpha_3$};
\draw (a) -- (b) -- (c) -- (d) -- (e) -- (f) -- (a) ;
\draw (e) -- (g) -- (b); 
\end{tikzpicture}
\end{center}
\end{minipage}\vline\begin{minipage}{0.5\textwidth}
\begin{center}
\begin{tikzpicture}[scale=0.65, crc/.style={circle,draw=black!50,thick,
inner sep=0pt,minimum size=8mm}, td/.style={rectangle,draw=black,thick, dotted,
inner sep=0pt,minimum size=8mm},  rrt/.style={circle,draw=red!50,thick,
inner sep=0pt,minimum size=8mm}, brt/.style={rectangle,draw=black,thick,
inner sep=0pt,minimum size=8mm}, transform shape]
 \node (a) at (1.73, 1) [td]  {$\alpha_{4}$};
 \node (b) at (90:2) [brt]  {$\alpha_{19}$};
  \node (c) at (150:2)   [td] {$\alpha_1$};
 \node (d) at (210:2) [brt]    {$\alpha_9$};
 \node (e) at (270:2)  [td]  {$\alpha_{14}$};
  \node (f) at (1.73, -1) [td]   {$\alpha_7$};
 \node (g) at (1.73+1.9, 1+0.62) [brt]  {$\alpha_{11}$};
 \node (h) at (1.73+1.9+1.175, 0) [td]  {$\alpha_{13}$};
 \node (i) at (1.73+1.9, -1-0.62) [brt]  {$\alpha_{10}$};
 \node (j) at (7.25, 0) [td]  {$\alpha_6$};
\draw (a) -- (b) -- (c) -- (d) -- (e) -- (f) -- (a) ;
\draw (a) -- (g) -- (h) -- (i) -- (f);
\draw [bend left] (b) to (j);
\draw [bend right] (i) to (j);
\end{tikzpicture}
\end{center}
\end{minipage}
\end{center}

\caption{The construction of $\cI$ and $\cJ$ for the two nonabelian cores of the 
form $[5, 12, 8]$ and $[6, 16, 12]$ of $\rE_6$, represented as in Figure \ref{fig:covcir}.}
\label{fig:covcor}
\end{figure}

\vspace{-2mm}

\begin{remark}
Let $\Phi$ be a root system of type $\rD_6$ or $\rE_6$. 
Let $\cI=\{i_1, \dots, i_m\}$, $\cJ=\{j_1, \dots, j_\ell\}$  and $\cZ$ satisfy the assumptions of Corollary \ref{cor:plus}. 
The equation $\lambda([y, x])=0$, for $x=x_{i_1}(t_{i_1}) \cdots x_{i_r}(t_{i_r})\in X$ and $y=x_{j_1}(s_{j_1}) \cdots x_{j_\ell}(s_{j_\ell}) \in Y$, where 
$t_{i_1}, \dots, t_{i_r}, 
s_{j_1}, \dots, s_{j_\ell}$ 
are unknown variables over $\F_q$, can be rewritten as 
\begin{equation} \label{eq:1i}
\sum_{h=1}^{\ell}\sum_{k=1}^{m} d_{h, k} s_{j_h} t_{i_k}=0,
\end{equation}
which just contains linear terms 
in the $s_{j_h}$'s and the $t_{i_k}$'s, where each constant $d_{h, k} \in \F_q$ depends on $\cS$ and the choice of the extraspecial pairs in $U$. In 
particular, if $\alpha_{i_h}+\alpha_{j_k} \notin \cS$ then $d_{h, k}=0$. 

It is then easy to work out explicitly $X'$ (respectively $Y'$), namely 
by finding the values of $t_{i_1}, \dots, t_{i_m} \in \F_q$ (respectively $s_{j_1}, \dots, s_{j_\ell} \in \F_q$) 
such that Equation \eqref{eq:1i} holds 
for every $s_{j_1}, \dots, s_{j_\ell} \in \F_q$ (respectively for every $t_{i_1}, \dots, t_{i_m} \in \F_q$). 
\end{remark}

\section{Parametrization of $\Irr(\mathrm{UD}_6(q))$ and $\Irr(\mathrm{UE}_6(q))$} 
\label{sec:para}
We now describe the 
parametrization of the sets $\Irr(X_{\cS})_{\cZ}$ 
that arise from nonabelian cores in types $\rD_6$ 
and $\rE_6$. By Theorem \ref{theo:iso}, it is 
enough to consider just one quattern arising 
from each $[z, m, c]$-core in Table \ref{tab:coretype}. Applying Propositions \ref{prop:HLM} and 
\ref{prop:small}, we can then obtain the corresponding 
parametrization of characters in $\Irr(U)$ using the information stored in $\cA$ 
and $\cL$ and the record of roots in direct products. 

The degrees of the irreducible characters and the numbers of 
irreducible characters of fixed degree arising from a 
nonabelian $[z, m, c]$-core are collected in Table \ref{tab:coresD6} for $\rU\rD_6(q)$ 
and in Table \ref{tab:coresE6} for $\rU\rE_6(q)$, along with the labels of the characters in each $\Irr(X_\cS)_\cZ$. We notice that the parametrization is 
uniform for $p \ge 3$ in type $\rD_6$, and for $p \ge 5$ in type $\rE_6$. For $q=2^f$ in type 
$\rD_6$, and for $q=2^f$ or $q=3^f$ in type $\rE_6$, the parametrization 
is more complicated.

Let us put $v:=q-1$, and let $\cS$ and $\cZ$ correspond to a 
nonabelian core $\fC$. The number 
$|\Irr(X_\cS)_\cZ|$ may not always be expressed 
as a polynomial in $v$ with nonnegative 
coefficients, even when $p$ is a good prime, as 
in the case of $\fC$ of the form $[7, 15, 9]$ in type $\rE_6$ (see Table \ref{tab:coresE6}), 
where $|\Irr(X_\cS)_\cZ|=v^6(v^2+2v-2)$ for corresponding 
$\cS$ and $\cZ$. 
Nevertheless, we notice that for a good prime $p$ in both cases $U=\rU\rD_6(q)$ and 
$U=\rU\rE_6(q)$, the character degrees of $U$ are powers of $q$, and the numbers $k(U, q^d)$ of irreducible characters of $U$ of degree $q^d$ are in $\mathbb{Z}[v]$  for every power $q$ of $p$ and every $d \in \Z_{\ge 0}$. 

We collect the numbers of irreducible 
characters of each fixed degree of $\rU\rD_6(q)$ in Tables \ref{tab:fam3D6} and 
\ref{tab:fam2D6} when $p \ge 3$ and $p=2$  
respectively, and of $\rU\rE_6(q)$ in Tables \ref{tab:fam5E6}, \ref{tab:fam3E6} and \ref{tab:fam2E6} 
when $p\ge 5$, $p=3$ and $p=2$ respectively. We notice that we 
have fractional degrees of the form $q^3/2, \dots, q^{11}/2$ and $q^{10}/4$ in 
$\rU\rD_6(2^f)$, $q^3/2, \dots, q^{15}/2$ 
in $\rU\rE_6(2^f)$ and $q^7/3$ in $\rU\rE_6(3^f)$. 
Observe that the numbers $k(U, D)$ of irreducible characters 
of $U$ of fixed degree $D$ can always be expressed as polynomials in 
$v$ with nonnegative coefficients; such coefficients are in fact 
integers, except in the cases 
$$k(\rU\rE_6(q), q^7) , k(\rU\rE_6(q), q^7/3)\in \mathbb{Z}[v/2] \setminus \mathbb{Z}[v].$$
When $p$ is a good prime, the formulas for each $k(U, q^d)$ coincide with 
the expressions obtained in \cite[Table 3]{GMR15} when $p$ is at least the Coxeter number of $G$. 

We explain the meaning of the labels in Tables \ref{tab:coresD6} and 
\ref{tab:coresE6}. Each character is inflated/induced from an abelian subquotient 
of $U$, which is a product of root subgroups and of diagonal subgroups 
of products of root subgroups as in Definition \ref{def:genroot}; these are indexed, respectively, by a number 
$i$ corresponding to the root $\alpha_i$, or a sequence $i_1, \dots, i_s$ with $i_1 < \dots < i_s$, 
corresponding to the set of roots $\alpha_{i_1}, \dots, \alpha_{i_s}$. The 
subgroups over which we inflate and induce are indexed by $\cA$, $\cK$ and 
the subgroups of the form $\tilde{X}$ and $\tilde{Y}$ at each application of 
Proposition \ref{prop:newrl} and Corollary \ref{cor:cornewrl}. We omit this 
information in the tables; it can be easily retrieved from the character labels 
in Tables \ref{tab:coresD6} and \ref{tab:coresE6}. 

As in $\S$\ref{sub:cores}, the label $\ua$ (respectively $\ub$) corresponds to a tuple of 
elements of $\F_q^\times$ (respectively $\F_q$). The meaning of 
$\ua^*$ in a label of the form $\ua, \ua^*$ is that $\ua^*_j \in \F_q^{\times} \setminus \{f_j(\ua)\}$ for 
every $1 \le j \le \ell$, 
where $\ell$ is a positive integer and each $f_j(\ua)$ is a nonzero expression depending on $\ua$; these are 
explicitly determined in each case. 
A label of 
the form $c^*$ corresponds to a more involved expression indexed by $\F_q^\times$, 
detailed in the case-by-case analysis. Finally, labels of the 
form $d$ index elements of a subset of $(\F_q,+)$ isomorphic to $(\F_p,+)$, and the 
labels $e^1$ and $e^2$ correspond respectively to $(q+1)/2$ and $(q-1)/2$ 
elements in $\F_q$ when $q=3^f$. 

\begin{longtabu}[ht]{|c|c|c|c|c|}
 \hline
Form & Family & Label & Number & Degree \\
\hline
\hline
\endhead
\hline
\endfoot
\hline \caption{Parametrization of 
nonabelian cores in $\rU\rD_6(q)$ for 
every $p$.} \label{tab:coresD6}
\endlastfoot
\multirow{2}{*}{$[3, 9, 6]$} & $\cF_1^{p \ne 2}$ & $\chi^{a_{18}, a_{19}, a_{24}}$ 
& $(q-1)^3$ & $q^3$ \\
\cline{2-5}
& $\cF_1^{p = 2}$ & $\chi_{b_{2, 10, 11}, b_{8, 14, 15}}^{a_{18}, a_{19}, a_{24}}$ 
& $q^2(q-1)^3$ & $q^2$ \\
\hline
\hline
\multirow{3}{*}{$[3, 10, 9]$} & $\cF_2^{p \ne 2}$ & $\chi_{b_2}^{a_{12}, a_{27}, a_{28}}$ 
& $q(q-1)^3$ & $q^3$ \\
\cline{2-5}
& $\cF_2^{1, p = 2}$ & $\chi^{a_{12}, a_{27}, a_{28}}$ 
& $(q-1)^3$ & $q^3$ \\
\cdashline{2-5}
& $\cF_2^{2, p = 2}$ & $\chi_{d_2, d_{1, 3, 24}}^{a_{7, 8, 26}, a_{12}, a_{27}, a_{28}}$ 
& $4(q-1)^4$ & $q^3/2$ \\
\hline
\hline
\multirow{5}{*}{$[4, 18, 18]$} & $\cF_3^{p \ne 2}$ & $\chi_{b_2, b_4}^{a_{16}, a_{21}, a_{22}, a_{28}}$ 
& $q^2(q-1)^4$ & $q^6$ \\
\cline{2-5}
&$\cF_3^{1, p = 2}$ & $\chi^{a_{16}, a_{21}, a_{22}, a_{28}}$ 
& $(q-1)^4$ & $q^6$ \\
\cdashline{2-5}
&$\cF_3^{2, p = 2}$ & $\chi_{d_{1, 14, 15}, d_2}^{a_{7, 18, 19}, a_{16}, a_{21}, a_{22}, a_{28}}$ 
& $4(q-1)^5$ & $q^6/2$ \\
\cdashline{2-5}
&$\cF_3^{3, p = 2}$ & $\chi_{d_{4}, d_{5, 6, 12}}^{a_{10, 11, 17}, a_{16}, a_{21}, a_{22}, a_{28}}$ 
& $4(q-1)^5$ & $q^6/2$ \\
\cdashline{2-5}
&$\cF_3^{4, p = 2}$ & $\chi_{d_{1, 14, 15}, d_2, d_{4}, d_{5, 6, 12}}^{a_{7, 18, 19}, a_{10, 11, 17}, a_{16}, a_{21}, a_{22}, a_{28}}$ 
& $16(q-1)^6$ & $q^6/4$ \\
\hline
\hline
\multirow{4}{*}{$[4, 21, 28]$} & $\cF_4^{p \ne 2}$ & $\chi_{b_{3, 13}, b_8, b_9}^{a_{20}, a_{21}, a_{22}, a_{26}}$ 
& $q^3(q-1)^4$ & $q^7$ \\
\cline{2-5}
&$\cF_4^{1, p = 2}$ & $\chi_{b_{3, 13}}^{a_{20}, a_{21}, a_{22}, a_{26}}$ 
& $q(q-1)^4$ & $q^7$ \\
\cdashline{2-5}
&$\cF_4^{2, p = 2}$ & $\chi_{b_{3, 9, 13}, d_{1, 10, 11}, d_8}^{a_{12, 18, 19}, a_{20}, a_{21}, a_{22}, a_{26}}$ 
& $4q(q-1)^5$ & $q^7/2$ \\
\cdashline{2-5}
&$\cF_4^{3, p = 2}$ & $\chi_{b_{3, 8, 13}, d_{5, 6, 7}, d_9}^{a_{14, 15, 17}, a_{20}, a_{21}, a_{22}, a_{26}}$ 
& $4q(q-1)^5$ & $q^7/2$ \\
\cdashline{2-5}
$[4, 21, 28]$ &$\cF_4^{4, p = 2}$ & $\chi_{d_{1, 5, 6, 7, 10, 11}, b_{3, 8, 9, 13}, d_{8, 9}}^{a_{12, 18, 19}, a_{14, 15, 17}, a_{20}, a_{21}, a_{22}, a_{26}}$ 
& $4q(q-1)^6$ & $q^7/2$ \\
\hline
\hline
\multirow{7}{*}{$[4, 24, 43]$} & $\cF_5^{1, p \ne 2}$ & $\chi_{b_3}^{a_{13}, a_{21}, a_{22}, a_{23}, a_{24}}$
& $q(q-1)^5$ & $q^9$ \\
\cdashline{2-5}
& $\cF_5^{2, p \ne 2}$ & $\chi^{a_{8, 9}, a_{21}, a_{22}, a_{23}, a_{24}}$
& $(q-1)^5$ & $q^9$ \\
\cdashline{2-5}
 & $\cF_5^{3, p \ne 2}$ & $\chi_{b_{2, 4}, b_3}^{a_{21}, a_{22}, a_{23}, a_{24}}$
& $q^2(q-1)^4$ & $q^8$ \\
\cline{2-5}
& $\cF_5^{1, p = 2}$ & $\chi_{d_{1, 5, 6}, b_3, d_{13}}^{a_{17, 18, 19}, a_{21}, a_{22}, a_{23}, a_{24}}$
& $4q(q-1)^5$ & $q^9/2$ \\
\cdashline{2-5}
& $\cF_5^{2, p = 2}$ & $\chi_{b_{2, 4, 7, 10, 11}, b_{12, 14, 15}}^{a_{8, 9}, a_{17, 18, 19}, a_{21}, a_{22}, a_{23}, a_{24}}$ 
& $q^2(q-1)^5$ & $q^8$ \\
\cdashline{2-5}
& $\cF_5^{3, p = 2}$ & $\chi_{b_{2, 4}}^{a_{21}, a_{22}, a_{23}, a_{24}}$
& $q(q-1)^4$ & $q^8$ \\
\cdashline{2-5}
& $\cF_5^{4, p = 2}$ & $\chi_{b_{2, 4}, d_3, d_{7, 11}}^{a_{12, 14, 15}, a_{21}, a_{22}, a_{23}, a_{24}}$ 
& $4q(q-1)^5$ & $q^8/2$ \\
\hline
\hline
\multirow{4}{*}{$[5, 18, 18]$} & $\cF_6^{p \ne 2}$ & $\chi_{b_3}^{a_{17}, a_{18}, a_{19}, a_{24}, a_{25}}$ 
& $q(q-1)^5$ & $q^6$ \\
\cline{2-5}
& $\cF_6^{1, p = 2}$ & $\chi_{b_{2, 4, 7, 10, 11, 16}, b_{9, 12, 20}}^{c_{8, 14, 15}^*, a_{17}, a_{18}, a_{19}, a_{24}, a_{25}}$ 
& $q^2(q-1)^6$ & $q^5$ \\
\cdashline{2-5}
& $\cF_6^{2, p = 2}$ & $\chi_{b_{2, 4, 7, 10, 11, 16}}^{a_{17}, a_{18}, a_{19}, a_{24}, a_{25}}$ 
& $q(q-1)^5$ & $q^5$ \\
\cdashline{2-5}
& $\cF_6^{3, p = 2}$ & $\chi_{b_{2, 4, 7, 10, 11, 16}, d_{2, 4, 7, 10, 11, 16}, d_3}^{c_{9, 12, 20}^*, a_{17}, a_{18}, a_{19}, a_{24}, a_{25}}$ 
& $4q(q-1)^6$ & $q^5/2$ \\
\hline
\hline
\multirow{4}{*}{$[6, 19, 20]$} & $\cF_7^{1, p \ne 2}$ & $\chi_{b_3}^{a_{13}, a_{17}, a_{18}, a_{19}, a_{24}, a_{25}^*}$ 
& $q(q-1)^5(q-2)$ & $q^6$ \\
\cdashline{2-5}
& $\cF_7^{2, p \ne 2}$ & $\chi^{a_{8, 9, 12, 14, 15, 20}, a_{13}, a_{17}, a_{18}, a_{19}, a_{24}}$ 
& $(q-1)^6$ & $q^6$ \\
\cdashline{2-5}
& $\cF_7^{3, p \ne 2}$ & $\chi_{b_{2, 4, 7, 10, 11, 16}, b_3}^{a_{13}, a_{17}, a_{18}, a_{19}, a_{24}}$ 
& $q^2(q-1)^5$ & $q^5$ \\
\cline{2-5}
& $\cF_7^{p = 2}$ & $\chi_{b_3}^{a_{13}, a_{17}, a_{18}, a_{19}, a_{24}, a_{25}}$ 
& $q(q-1)^6$ & $q^6$ \\
\end{longtabu}

\begin{longtabu}[ht]{|c|c|c|c|c|}
 \hline
Form & Family & Label & Number & Degree \\
\hline
\hline
\endhead
\hline
\endfoot
\hline \caption{Parametrization of 
nonabelian cores in $\rU\rE_6(q)$ for 
every $p$.} \label{tab:coresE6}
\endlastfoot
\multirow{2}{*}{$[3, 9, 6]$} & $\cF_1^{p \ne 2}$ & $\chi^{a_{23}, a_{29}, a_{31}}$ 
& $(q-1)^3$ & $q^3$ \\
\cline{2-5}
& $\cF_1^{p = 2}$ & $\chi_{b_{7, 11, 19}, b_{12, 16, 24}}^{a_{23}, a_{29}, a_{31}}$ 
& $q^2(q-1)^3$ & $q^2$ \\
\hline
\hline
\multirow{3}{*}{$[3, 10, 9]$} & $\cF_2^{p \ne 2}$ & $\chi_{b_4}^{a_{23}, a_{29}, a_{31}}$ 
& $q(q-1)^3$ & $q^3$ \\
\cline{2-5}
& $\cF_2^{1, p = 2}$ & $\chi^{a_{23}, a_{29}, a_{31}}$ 
& $(q-1)^3$ & $q^3$ \\
\cdashline{2-5}
& $\cF_2^{2, p = 2}$ & $\chi_{d_4, d_{7, 11, 19}}^{a_{12, 16, 24}, a_{23}, a_{29}, a_{31}}$ 
& $4(q-1)^4$ & $q^3/2$ \\
\hline
\hline
\multirow{2}{*}{$[4, 8, 4]$} & $\cF_3^1$ & $\chi^{a_8, a_{12}, a_{14}, a_{18}^*}$ 
& $(q-1)^3(q-2)$ & $q^2$ \\
\cdashline{2-5}
& $\cF_3^2$ & $\chi_{b_{2, 7}, b_{4, 10}}^{a_8, a_{12}, a_{14}}$ 
& $q^2(q-1)^3$ & $q$ \\
\hline
\hline
$[5, 10, 5]$ & $\cF_4$ & $\chi_{b_{1, 4, 13}}^{a_7, a_9, a_{10}, a_{17}, a_{19}}$ 
& $q(q-1)^5$ & $q^2$ \\
\hline
\hline
$[5, 12, 8]$ & $\cF_5$ & $\chi_{b_{3, 4, 6, 17}}^{a_9, a_{15}, a_{16}, a_{26}, a_{27}}$ 
& $q(q-1)^5$ & $q^3$ \\
\hline
\hline
\multirow{2}{*}{$[5, 15, 11]$} & $\cF_6^{p \ne 3}$ & $\chi^{a_{12}, a_{16}, a_{22}, a_{24}, a_{25}}$ 
& $(q-1)^5$ & $q^5$ \\
\cline{2-5}
& $\cF_6^{p =3}$ & $\chi_{b_{1, 4, 6, 13, 14}, b_{7, 9, 10, 11, 19}}^{a_{12}, a_{16}, a_{22}, a_{24}, a_{25}}$ 
& $q^2(q-1)^5$ & $q^4$ \\
\hline
\hline
\multirow{3}{*}{$[5, 16, 15]$} & $\cF_7^{p \ne 2}$ & $\chi_{b_4}^{a_{15}, a_{17}, a_{18}, a_{20}, a_{21}}$ 
& $q(q-1)^5$ & $q^5$ \\
\cline{2-5}
& $\cF_7^{1, p = 2}$ & $\chi^{a_{15}, a_{17}, a_{18}, a_{20}, a_{21}}$ 
& $(q-1)^5$ & $q^5$ \\
\cdashline{2-5}
& $\cF_7^{2, p = 2}$ & $\chi_{d_4, d_{2, 3, 5, 7, 11}}^{a_{8, 9, 10, 12, 16}, a_{15}, a_{17}, a_{18}, a_{20}, a_{21}}$ 
& $4(q-1)^6$ & $q^5/2$ \\
\hline
\hline
\multirow{2}{*}{$[5, 20, 25]$} & $\cF_8^{p \ne 3}$ & $\chi_{b_{2, 3, 5}}^{a_{17}, a_{18}, a_{20}, a_{21}, a_{24}}$ 
& $q(q-1)^5$ & $q^7$ \\
\cline{2-5}
& $\cF_8^{1, p = 3}$ & $\chi^{a_{7, 11, 13, 14, 15}, a_{17}, a_{18}, a_{20}, a_{21}, a_{24}}$ 
& $(q-1)^6$ & $q^7$ \\
\cdashline{2-5}
$[5, 20, 25]$ & $\cF_8^{2, p = 3}$ & $\chi_{b_{1, 6, 8, 9, 10, 12, 16}, b_{2, 3, 5}}^{ a_{17}, a_{18}, a_{20}, a_{21}, a_{24}}$ 
& $q^2(q-1)^5$ & $q^6$ \\
\hline
\hline
\multirow{4}{*}{$[5, 21, 30]$} & $\cF_9^{p \ne 3}$ & $\chi_{b_4, b_{8, 9, 10}}^{a_{17}, a_{18}, a_{19},  a_{20}, a_{21}}$ 
& $q^2(q-1)^5$ & $q^7$ \\
\cline{2-5}
& $\cF_9^{1, p = 3}$ & $\chi_{b_4}^{a_{12, 13, 14, 15, 16}, a_{17}, a_{18}, a_{19},  a_{20}, a_{21}}$ 
& $q(q-1)^6$ & $q^7$ \\
\cdashline{2-5}
 & $\cF_9^{2, p = 3}$ & $\chi^{ e_{8, 9, 10}^1, a_{17}, a_{18}, a_{19},  a_{20}, a_{21}}$ 
& $(q-1)^5(q+1)/2$ & $q^7$ \\
\cdashline{2-5}
& $\cF_9^{3, p = 3}$ & $\chi_{d_{1, 2, 3, 5, 6, 7, 11}, d_{4}}^{ e_{8, 9, 10}^2, a_{17}, a_{18}, a_{19},  a_{20}, a_{21}}$ 
& $9(q-1)^6/2$ & $q^7/3$ \\
\hline
\hline
\multirow{2}{*}{$[6, 12, 6]$} & $\cF_{10}^1$ & $\chi^{a_8, a_{10}, a_{12}, a_{15}, a_{23}, a_{25}^*}$ 
& $(q-1)^5(q-2)$ & $q^3$ \\
\cdashline{2-5}
& $\cF_{10}^2$ & $\chi_{b_{1, 2, 5}, b_{4, 9, 21}}^{a_8, a_{10}, a_{12}, a_{15}, a_{23}}$ 
& $q^2(q-1)^5$ & $q^2$ \\
\hline
\hline
$[6, 13, 7]$ & $\cF_{11}$ & $\chi_{b_{1, 5, 14, 24}}^{a_7, a_{11}, a_{18}, a_{19}, a_{26}, a_{28}}$ 
& $q(q-1)^6$ & $q^3$ \\
\hline
\hline
\multirow{2}{*}{$[6, 14, 8]$} & $\cF_{12}^1$ & $\chi^{a_{12}, a_{13}, a_{15}, a_{16}, a_{20}, a_{22}^*}$ 
& $(q-1)^5(q-2)$ & $q^4$ \\
\cdashline{2-5}
& $\cF_{12}^2$ & $\chi_{b_{3, 6, 7, 11}, b_{4, 8, 10, 14}}^{a_{12}, a_{13}, a_{15}, a_{16}, a_{20}}$ 
& $q^2(q-1)^5$ & $q^3$ \\
\hline
\hline
\multirow{3}{*}{$[6, 15, 12]$} & $\cF_{13}^{p \ne 2}$ & $\chi_{b_{4, 10}}^{a_8, a_9, a_{15}, a_{20}, a_{22}, a_{23}}$ 
& $q(q-1)^6$ & $q^4$ \\
\cline{2-5}
& $\cF_{13}^{1, p = 2}$ & $\chi^{a_8, a_9, a_{15}, a_{20}, a_{22}, a_{23}}$ 
& $(q-1)^6$ & $q^4$ \\
\cdashline{2-5}
& $\cF_{13}^{2, p = 2}$ & $\chi_{d_{2, 3, 6, 7}, d_{4, 10}}^{a_8, a_9, c_{14, 16, 18}^*, a_{15}, a_{20}, a_{22}, a_{23}}$ 
& $4(q-1)^7$ & $q^4/2$ \\
\hline
\hline
\multirow{2}{*}{$[6, 16, 12]$} & $\cF_{14}^{p \ne 3}$ & $\chi^{a_{12}, a_{16}, a_{18}, a_{22}, a_{24}, a_{25}}$ 
& $(q-1)^6$ & $q^5$ \\
\cline{2-5}
& $\cF_{14}^{p = 3}$ & $\chi_{b_{1, 6, 7, 14}, b_{4, 6, 7, 13}}^{a_{12}, a_{16}, a_{18}, a_{22}, a_{24}, a_{25}}$ 
& $q^2(q-1)^6$ & $q^4$ \\
\hline
\hline
\multirow{7}{*}{$[6, 17, 17]$}& $\cF_{15}^{1, p \ne 2}$ & $\chi_{b_4}^{a_{13}, a_{14}, a_{15}, a_{17}, a_{20}, a_{23}^*}$ 
& $q(q-1)^5(q-2)$ & $q^5$ \\
\cdashline{2-5}
& $\cF_{15}^{2, p \ne 2}$ & $\chi^{a_{8, 9, 10, 12, 16}, a_{13}, a_{14}, a_{15}, a_{17}, a_{20}}$ 
& $(q-1)^6$ & $q^5$ \\
\cdashline{2-5}
& $\cF_{15}^{3, p \ne 2}$ & $\chi_{b_{2, 3, 5, 7, 11}, b_4}^{a_{13}, a_{14}, a_{15}, a_{17}, a_{20}}$ 
& $q^2(q-1)^5$ & $q^4$ \\
\cline{2-5}
& $\cF_{15}^{1, p = 2}$ & $\chi^{a_{8, 9, 10, 12, 16}, a_{13}, a_{14}, a_{15}, a_{17}, a_{20}}$ 
& $(q-1)^6$ & $q^5$ \\
\cdashline{2-5}
& $\cF_{15}^{2, p = 2}$ & $\chi^{a_{13}, a_{14}, a_{15}, a_{17}, a_{20}, a_{23}^*}$ 
& $(q-1)^5(q-2)$ & $q^5$ \\
\cdashline{2-5}
& $\cF_{15}^{3, p = 2}$ & $\chi_{d_{2, 3, 5, 7, 11}, d_4}^{a_{8, 9, 10, 12, 16}, a_{13}, a_{14}, a_{15}, a_{17}, a_{20}, a_{23}^*}$ 
& $4(q-1)^6(q-2)$ & $q^5/2$ \\
\cdashline{2-5}
& $\cF_{15}^{4, p = 2}$ & $\chi_{b_{2, 3, 5, 7, 11}, b_4}^{a_{13}, a_{14}, a_{15}, a_{17}, a_{20}}$ 
& $q^2(q-1)^5$ & $q^4$ \\
\hline
\hline
\multirow{3}{*}{$[7, 15, 9]$} & $\cF_{16}^1$ & $\chi^{a_9, a_{12} a_{13}, a_{15}, a_{16}, a_{20}^*, a_{22}^*}$ 
& $(q-1)^5(q-2)^2$ & $q^4$ \\
\cdashline{2-5}
& $\cF_{16}^2$ & $\chi^{a_9, a_{12} a_{13}, a_{15}, a_{16}, a_{22}}$ 
& $(q-1)^6$ & $q^4$ \\
\cdashline{2-5}
& $\cF_{16}^3$ & $\chi_{b_{3, 6, 7, 11}, b_{4, 8, 10, 14}}^{a_9, a_{12} a_{13}, a_{15}, a_{16}, a_{20}^*}$ 
& $q^2(q-1)^5(q-2)$ & $q^3$ \\
\end{longtabu}

\subsection{Heartless cores} \label{sub:heartless} 
Recall that the heart of a nonabelian core consists of the roots $\alpha$ such that 
$\{\alpha\}$ is a connected component in the associated graph defined in \S\ref{sub:graph}.
Among the cores of $\rD_6$ (respectively $\rE_6$) listed in Table \ref{tab:coresD6} (respectively Table \ref{tab:coresE6}), 
the following forms are heartless, 
$$[3, 9, 6], \quad
[4, 8, 4], \quad
[5, 10, 5], \quad
[5, 12, 8], \quad
[5, 15, 11], $$
$$[6, 12, 6], \quad
[6, 13, 7], \quad
[6, 14, 8],\quad
[6, 16, 12], \quad
[7, 15, 9].$$
In Proposition \ref{core[6,16,12]} we study in detail the $[6, 16, 12]$-core, whose 
analysis departs from the uniform treatment of the remaining heartless cores which we discuss 
first. 

Now let $(\cS, \cZ, \cA, \cL, \cK)$ be a nonabelian core of one of the forms listed above, but not $[6,16,12]$. 
Then we have that $\cS=\cZ \cup \cI \cup \cJ$, for $\cI=\{i_1, \dots, i_k\}$ and $\cJ=\{j_1, \dots, j_h\}$ 
as defined in Section \ref{sec:cors}, and the study of Equation \eqref{eq:1i} yields $X'=1$ or $X'=\{x_{i_1, \dots, i_k}(t) \mid t \in \F_q\}$, 
and $Y'=1$ or $Y'=\{x_{j_1, \dots, j_h}(s) \mid s \in \F_q\}$ for some $x_{i_1, \dots, i_k}(t)$ and $x_{j_1, \dots, j_h}(s)$ 
as in Definition \ref{def:genroot}. 
Hence by Proposition \ref{prop:newrl} and Equation \eqref{eq:decomp}, we have that 
$$\Irr(X_{\cS})_{\cZ}=\bigsqcup_{\lambda \in \Irr(X_{\cZ})} \Irr(X'Y'Z \mid \lambda),$$
where $Z=X_{\cZ}/(\ker \lambda)$. 

Then we have 
$$\Irr(X_{\cS})_{\cZ}=\{\chi_{\ub}^{\ua} \mid \ub \in \F_q^\delta, \ua \in (\F_q^\times)^{|\cZ|}\} \,\, \text{or} \,\, \Irr(X_{\cS})_{\cZ}=\{\chi_{\ub}^{\ua, \ua^*} \mid \ub \in \F_q^\delta, \ua \in (\F_q^\times)^{|\cZ|-\ell}, \ua^* \in S\},$$
where $\delta=\log_q(|X'Y'|) \in \{0, 1, 2\}$, and $\ua$ is indexed by root indices of $X_{\cZ}$ and $\ub$ 
is indexed by the $i_1, \dots, i_k$ or $j_1, \dots, j_h$ 
whenever $X' \ne 1$ or $Y' \ne 1$, and 
$$S:=(\F_q^\times \setminus \{f_1(\ua)\}) \times \cdots \times (\F_q^\times \setminus \{f_{\ell}(\ua)\})$$
for an integer $\ell \in \{0, 1, 2\}$ and fractional polynomial expressions $f_1, \ldots, f_{\ell}$ that are 
explicitly determined. The graph structure of a heartless core $\fC$ is easily determined by $\cS$. 

The cores of the form $[3, 9, 6]$ in both $\rD_6$ and $\rE_6$ are isomorphic to one of the 
nonabelian cores of $\rF_4$ determined in \cite[\S4.3]{GLMP16}. 
As done in \cite{GLMP16}, we include no further details here for the straightforward analysis of heartless cores, and we refer 
to \cite{LMP} for information on the sets $\cS$, $\cZ$, $\cA$ and $\cL$ 
and the form of Equation \eqref{eq:1i} in each case. 
The labels of the sets $\Irr(X_{\cS})_{\cZ}$ are collected in Tables \ref{tab:coresD6} and 
\ref{tab:coresE6}. 

Recall that there is a unique nonabelian core of $\rE_6$ of the form $[6, 16, 12]$. In this case, 
\begin{itemize}
\item $\cS=\{\alpha_{1}, \alpha_{4}, \alpha_{6}, \alpha_{7}, \alpha_{9}, \alpha_{10}, \alpha_{11}, \alpha_{12}, \alpha_{13}, \alpha_{14}, \alpha_{16}, \alpha_{18}, \alpha_{19}, \alpha_{22}, \alpha_{24}, \alpha_{25}\}$, 
\item $\cZ=\{\alpha_{12}, \alpha_{16}, \alpha_{18}, \alpha_{22}, \alpha_{24}, \alpha_{25}\}$, 
\item $\cA=\{\alpha_{2}, \alpha_{3}, \alpha_{5}, \alpha_{15}\}$ and 
$\cL=\{\alpha_{8}, \alpha_{17}, \alpha_{20}, \alpha_{21}\}$, 
\item $\cI=\{\alpha_{1}, \alpha_{4}, \alpha_{6}, \alpha_{7}, \alpha_{13}, \alpha_{14}\}$ and $\cJ=\{\alpha_{9}, \alpha_{10}, \alpha_{11}, \alpha_{19}\}$.
\end{itemize}
Our analysis differs from the previous cases in that $X'$ is not a subgroup, and $|X'|=q^2$. 
The graph structure of $\fC$, represented in Figure \ref{fig:covcor}, is more complicated than in the case of the other heartless cores, 
as we have $7$ circles of both parities. This case can be examined in a similar way by applying Proposition \ref{prop:newrl} 
after the first reduction. We end this subsection by including the computational details in this case. 

\begin{prop}
\label{core[6,16,12]}
The irreducible characters corresponding to the $[6, 16, 12]$-core in type $\rE_6$ are parametrized as follows: 
\begin{itemize}
\item 
If $p \ne 3$, then $\Irr(X_{\cS})_{\cZ}=\cF_{14}^{p \ne 3}$ 
consists 
of $(q-1)^6$ characters of degree $q^5$. 
\item 
If $p=3$, then $\Irr(X_{\cS})_{\cZ}=\cF_{14}^{p = 3}$ 
consists 
of $q^2(q-1)^6$ characters of degree $q^4$. 
\end{itemize}
The labels of the characters in $\cF_{14}^{p \ne 3}$ and 
$\cF_{14}^{p = 3}$ are collected in Table \ref{tab:coresE6}. 
\end{prop}

\begin{proof}
The form of Equation \eqref{eq:1i} is 
\begin{align*}
 s_1( a_{22}t_{19}+a_{12}t_9 ) 
&+s_4 (a_{16}t_{11}+a_{24}t_{19} ) 
+s_6 (-a_{16}t_{10}-a_{25}t_{19} ) 
+s_7 ( a_{18}t_{10} ) \\
&+s_{13}( a_{24}t_{10}+a_{25}t_{11} )
+s_{14}( a_{24}t_9 ) =0.
\end{align*}

We have $X'=X_1'X_2'$ with $X_1':=\{x_{1, 6, 7, 14}(t_1) \mid t_1 \in \F_q\}$ and $X_2:=\{x_{4, 6, 7, 13}(t_2) \mid t_2 \in \F_q\}$, and $Y'=1$, where 
$$x_{1, 6, 7, 14}(t_1):=x_{1}(a_{18}a_{24}a_{25}t_1)x_{6}(a_{18}a_{22}a_{24}t_1)x_{7}(a_{16}a_{22}a_{24}t_1)x_{14}(-a_{12}a_{18}a_{25}t_1)
$$and$$
x_{4, 6, 7, 13}(t_2):=x_{4}(a_{18}a_{25}t_2)x_{6}(a_{18}a_{24}t_2)x_{7}(2a_{16}a_{24}t_2)x_{13}(-a_{16}a_{18}t_2).$$
We notice that each of $X_1'$ and $X_2'$ are subgroups, but we have 
$$[x_{1, 6, 7, 14}(t_1), x_{4, 6, 7, 13}(t_2)]
=x_{12}(a_{16}a_{18}a_{22}a_{24}a_{25}t_1t_2)
x_{16}(2a_{12}a_{18}a_{22}a_{24}a_{25}t_1t_2),$$
hence 
$$\lambda([x_{1, 6, 7, 14}(t_1), x_{4, 6, 7, 13}(t_2)])=\phi(3a_{12}a_{16}a_{18}a_{22}a_{24}a_{25}t_1t_2)$$
and $X'$ is not necessarily a subgroup of $X_{\cS}$. 

If $p \ne 3$, then we can apply again Proposition \ref{prop:newrl}  
with arm $X_1'$ and leg $X_2'$, reducing to the 
abelian subquotient $X_{\cZ}/(\ker \lambda)$. This gives the 
family $\cF_{14}^{p \ne 3}$ in Table \ref{tab:coresE6}. 

If $p=3$, then $X'$ and 
$X'X_{\cZ}$ are abelian subgroups of $X_{\cS}/(\ker \lambda)$. 
In this case, we obtain the family $\cF_{14}^{p = 3}$ in Table \ref{tab:coresE6}, which 
concludes our analysis. 
\end{proof}

\subsection{Cores with a heart}\label{sub:chea} We now want to 
investigate the cores of the form  
$$[3, 10, 9], \,\,
[4, 18, 18], \,\,
[4, 21, 28], \,\,
[4, 24, 43], \,\,
[5, 18, 18], \,\,
[6, 19, 20]$$
in type $\rD_6$, and of the form 
$$[3, 10, 9], \,\,
[5, 16, 15], \,\,
[5, 20, 25], \,\,
[5, 21, 30], \,\,
[6, 15, 12], \,\,
[6, 17, 17]$$
in type $\rE_6$. To study them, we apply 
repeatedly Proposition \ref{prop:newrl} 
and Corollary \ref{cor:cornewrl}. 

We recall the structure of the graph of $\fC$ in each case. The cores of 
the form $[3, 10, 9]$ have a single circle, and a heart of size $1$. The graphs 
of the cores of the form $[4, 18, 18], [4, 21, 28]$ and $[5, 18, 18]$ have 
two connected components, and each of them is a hexagon. 
Their hearts have sizes $2$, $5$ and $1$ respectively. The 
graph of the $[6, 19, 20]$-core contains $6$ circles, as Figure \ref{fig:exgra} shows, and its heart has 
size $1$. The graph of $[5, 16, 15]$-core, as in Figure \ref{fig:exgra}, has three circles, as well as 
the graphs of the $[5, 20, 25]$-core and the $[5, 21, 30]$-core, and their hearts have sizes 
$1$, $5$ and $6$ respectively. The heart of the latter nonabelian core is the biggest among 
all nonabelian cores in rank $6$ or less, which makes it one of the most complicated to study. 
We just refer to \cite[Section 3]{LM15} in the sequel, where the study of the $[5, 21, 30]$-core of $\rE_6$ has 
been carried out thoroughly. The graph of $[6, 15, 12]$-core has two connected components, 
namely its unique circle, and its heart of size $1$. Finally, we find $3$ circles in the $[6, 17, 17]$-core; 
here $|\cH|=1$. 

The cores of the form $[3, 10, 9]$ in types $\rD_6$ and $\rE_6$ are 
isomorphic to the only $[3, 10, 9]$-core in type $\rF_4$. 
The additional complication in the analysis of a 
core with a heart, as in \cite[\S4.3]{GLMP16}, lies in the determination of 
a certain non-linear polynomial over $\F_q$, which arises from the action of the root subgroups indexed by 
the heart on a suitable subquotient of $X_{\cS}$, and of the solutions in $\F_q$ of an 
equation depending on such a polynomial and the function $\phi:\F_q \to \C^\times$ defined in \S\ref{ss:characters}. 
The typical situation is that we get a polynomial of degree $p$ when $p$ is a 
bad prime for $G$. If $p=2$, then the 
situation can be easily described. 

\begin{remark}\label{rem:quad}
Let $q=2^f$, and let us consider the following expression in $\F_q$, 
$$f(s, t)=\phi(st(b+at))$$
for every $b, a \in \F_q$. Let us define
$$Z_1:=\{s \in \F_q \mid f(s, t)=0 \text{ for all } t \in \F_q\}, \quad 
Z_2:=\{t \in \F_q \mid f(s, t)=0 \text{ for all } s \in \F_q\}.
$$
It is easy to see that
\begin{itemize}
\item If $b=a=0$, then $Z_1=Z_2=\F_q$.
\item If ($b\ne 0$ and $a = 0$) or ($b= 0$ and $a \ne 0$), then $Z_1=Z_2=1$.
\item If $b\ne 0$ and $a \ne 0$, then $Z_1=\{0, a/b^2\}$ and $Z_2=\{0, b/a\}$.
\end{itemize}

\end{remark}

When $p=3$ is a bad prime for $\rE_6$, the polynomial arising from the above investigation 
is of degree $3$ just in the case of the core of the form $[5, 21, 30]$, 
which gives rise to irreducible character degrees $q^7/3$ in $\rU\rE_6(3^f)$; 
as previously remarked, 
the study of this core 
is detailed in 
\cite[Section 3]{LM15}. 
We include below the analysis just for 
three nonabelian cores with a heart. We discuss first the core of the form $[4, 18, 18]$ in 
$\rD_6$, which for $p=2$ gives rise to the only examples of irreducible characters of $\rU\rY_r(q)$ of the form 
$q^m/p^i$ with $i \ge 2$ when $\rY$ is of simply laced type and $r \le 6$. We then include 
full details for the $[4, 24, 43]$-core in type $\rD_6$ and the $[5, 20, 25]$-core 
in type $\rE_6$; one notices the different behavior of the 
bad primes $p=2$ in type $\rD_6$ and $p=3$ in type $\rE_6$ respectively. 
We decide to include every step in their study since these two cores, 
along with the $[5, 21, 30]$-core, seem to be the most difficult cases to examine. 
The other 
cases are investigated in a similar manner; the computations are available in 
\cite{LMP}. All character labels for each nonabelian core with a heart are also collected in 
Tables \ref{tab:coresD6} and \ref{tab:coresE6}. 

We use the following notation for a core 
$\fC$ corresponding to $\cS$ and 
$\cZ$. We construct subquotients 
$V_{\ub}^{n}$, 
and  $H_{\ub}^n$, $X_{\ub}^n$,  $Y_{\ub}^n$, 
$X_{\ub}^{' n}$,  $Y_{\ub}^{' n}$ 
and the character $\lambda^{\ub} $ 
in the following way. The index $n$ corresponds to 
the $n$-th application of Proposition \ref{prop:newrl} (possibly trivial if we just enlarge the kernel of a central character 
from step $n-1$ to step $n$), and $\ub$ denotes a certain tuple with entries in $\F_q$, which corresponds to the 
value of a central character. 

We initialize $\ub=\emptyset$ the empty tuple and  $V_{\emptyset}^0=X_{\cS}$,  
and $\lambda^{\emptyset}=\lambda$ a central character. Now assume $V_{\uhb}^{n-1}$ 
is constructed for $n \ge 1$. We assume 
that Proposition \ref{prop:newrl} applies to 
$V=V_{\uhb}^{n-1}$
and let $H=:H_{\uhb}^{n}$, $X=:X_{\uhb}^{n}$, 
$Y=:Y_{\uhb}^{n}$, $X'=:X_{\uhb}^{' n}$,
$Y'=:Y_{\uhb}^{' n}$. 
Finally, for every $\tilde{\ub}=(\tilde{b_1},  \dots, 
\tilde{b_s}) \in \F_q^s$ and $\mu_{\tilde{\ub}}$ as in Corollary \ref{cor:cornewrl}, we define $\ub$ as the concatenation of $\uhb$ and $\tilde{\ub}$ and 
$\lambda^{\ub}:=\lambda^{\uhb} \otimes \mu_{\tilde{\ub}}$, 
and we 
construct 
$V_{\ub}^{n}:=(V_{\uhb}^{n-1})_{\tilde{\ub}}$
. By convention, 
in the sequel we omit the top index $1$, and we drop the symbol $\emptyset$ when it occurs. 

We start by studying the unique core of the form $[4, 18, 18]$ in 
type $\rD_6$. In this case, 
\begin{itemize}
\item $\cS=\{\alpha_{1}, \alpha_{2}, \alpha_{4}, \alpha_{5}, \alpha_{6}, \alpha_{7}, \alpha_{10}, \alpha_{11}, \alpha_{12}, \alpha_{14}, \alpha_{15}, \alpha_{
16}, \alpha_{17}, \alpha_{18}, \alpha_{19}, \alpha_{21}, \alpha_{22}, \alpha_{28}\}$, 
\item $\cZ=\{\alpha_{16}, \alpha_{21}, \alpha_{22}, \alpha_{28}\}$, 
\item $\cA=\{\alpha_{3}, \alpha_{8}, \alpha_{9}, \alpha_{13}\}$ and 
$\cL=\{\alpha_{20}, \alpha_{23}, \alpha_{24}, \alpha_{26}\}$, 
\item $\cI=\{ \alpha_{1}, \alpha_{5}, \alpha_{6}, \alpha_{12}, \alpha_{
14}, \alpha_{15} \}$ and $\cJ=\{\alpha_7, \alpha_{10}, \alpha_{11}, \alpha_{17}, \alpha_{18}, \alpha_{19}\}$.
\end{itemize}

\begin{prop}
\label{core[4,18,18]}
The irreducible characters corresponding to the $[4, 18, 18]$-core in type $\rD_6$ are parametrized as follows: 

\begin{itemize}
\item If $p\neq 2$, then $\Irr(X_{\cS})_{\cZ}=\cF_3^{p \ne 2}$ consists of $q^2(q-1)^4$ 
characters of degree $q^6$. 

\item 
If $p=2$, then 
$$\Irr(X_{\cS})_{\cZ}=\cF_3^{1, p=2} \sqcup \cF_3^{2, p=2} \sqcup\cF_3^{3, p=2} \sqcup\cF_4^{4, p=2},$$
where
  \begin{itemize}
    \item $\cF_3^{1, p=2}$ consists of $(q-1)^4$ characters of degree $q^6$, 
    \item $\cF_3^{2, p=2}$ and $\cF_3^{3, p=2}$ consist each of $4(q-1)^5$ characters of degree $q^6/2$, and 
    \item $\cF_3^{4, p=2}$ consists of $16(q-1)^6$ characters of degree $q^6/4$. 
  \end{itemize}
\end{itemize}
The labels of the characters in $\cF_3^{p \ne 2}$ and in $\cF_3^{1, p=2}$, $\dots$, $\cF_3^{4, p=2}$ are collected in Table \ref{tab:coresD6}. 
\end{prop}

\begin{proof}
The form of Equation \eqref{eq:1i} is 
\begin{align*}
s_7(a_{21}t_{14}+a_{22}t_{15})&+s_{10}(-a_{16}t_6-a_{21}t_{12})+s_{11}(-a_{16}t_5-a_{22}t_{12})+s_{17}(a_{21}t_5+a_{22}t_6)+\\
&+s_{18}(-a_{21}t_1-a_{28}t_{15})+s_{19}(-a_{22}t_1-a_{28}t_{14})=0.
\end{align*}

If $p \ne 2$, then $X'=Y'=1$, and $\overline{V}=X_2X_4Z/(\ker \lambda)$ is abelian. We obtain the family $\cF_3^{p \ne 2}$ in Table \ref{tab:coresD6}. 

If $p=2$, we 
have that $X':=X_1'X_2'$ and $Y_1'Y_2'$, where
$$X_1':=\{x_{1, 14, 15}(t_1) \mid t_1 \in \F_q\} \quad \text{and} \quad X_1':=\{x_{5, 6, 12}(t_2) \mid  t_2 \in \F_q\},$$
$$Y_1':=\{x_{7, 18, 19}(s_1) \mid s_1 \in \F_q\} \quad \text{and} \quad Y_2':=\{x_{10, 11, 17}(s_2) \mid  s_2 \in \F_q\},$$
and for every $s_1, s_2, t_1, t_2 \in \F_q$, 
$$x_{1, 14, 15}(t_1):=x_{1}(a_{28}t_1)x_{14}(a_{22}t_1)x_{15}(a_{21}t_1) \quad \text{and} \quad
x_{5, 6, 12}(t_2):=x_{5}(a_{22}t_2)x_{6}(a_{21}t_2)x_{12}(a_{16}t_2)$$
$$x_{7, 18, 19}(s_1):=x_{7}(a_{28}s_1)x_{18}(a_{22}s_1)x_{19}(a_{21}s_1) \ \ \text{and} \ \ x_{10, 11, 17}(s_2):=x_{10}(a_{22}s_2)x_{11}(a_{21}s_2)x_{17}(a_{16}s_2).$$
Notice that $X'$ is a subgroup of $\overline{V}$. 
We extend $\lambda$ to $\lambda'=\lambda^{c_{7, 18, 19}, c_{10, 11, 17}}$ for every $c_{7, 18, 19}, c_{10, 11, 17} \in \F_q$. In $\overline{V}$, we have that $[X_1', X_4]=[X_2', X_2]=1$, and that 
\begin{align*}
[x_2(&s_2)x_4(s_4), x_{1, 14, 15}(t_1)x_{5, 6, 12}(t_2)]=x_{7, 18, 19}(s_2t_1)x_{10, 11, 17}(s_4t_2)x_{16}(a_{21}a_{22}s_4t_2^2) \, \cdot \\
\cdot \,  
&x_{21}(a_{22}a_{28}s_2t_1^2+a_{16}a_{22}s_4t_2^2)x_{22}(a_{21}a_{28}s_2t_1^2+a_{16}a_{21}s_4t_2^2)x_{28}(a_{21}a_{22}s_2t_1^2). 
\end{align*}
We then want to apply Proposition \ref{prop:newrl} with $X'$ as a candidate for an arm, 
and $X_2X_4$ as a candidate for a leg. We 
apply $\lambda$ to the above, and we use 
Remark \ref{rem:quad} study 
the equation 
$$\phi(
s_2t_1(c_{7, 18, 19}+a_{21}a_{22}a_{28}t_1)
+s_4t_2(c_{10, 11, 17}+a_{16}a_{21}a_{22}t_2)
)=1.$$

If $c_{7, 18, 19}=0$ and $c_{10, 11, 17}=0$, 
then $X_{(0, 0)}^{' 2}=Y_{(0, 0)}^{' 2}=1$ and 
$V^{2}_{(0, 0)}$ 
is abelian. This gives the family $\cF_3^{1, p = 2}$ in Table \ref{tab:coresD6}. 

If $a_{7, 18, 19}:=c_{7, 18, 19}\ne 0$ and $c_{10, 11, 17}=0$, then 
$$
X_{(a_{7, 18, 19}, 0)}^{' 2}:=\{1, x_{1, 14, 15}(a_{7, 18, 19}/(a_{21}a_{22}a_{28}))\} \ \ \text{and} \ \ 
Y_{(a_{7, 18, 19}, 0)}^{' 2}:=\{1, x_2(a_{21}a_{22}a_{28}/(a_{7, 18, 19}^2))\},$$
and 
$V^{2}_{(a_{7, 18, 19}, 0)} $ 
is abelian. This gives the family $\cF_3^{2, p = 2}$ in Table \ref{tab:coresD6}. 

If $c_{7, 18, 19}= 0$ and $a_{10, 11, 17}:=c_{10, 11, 17} \ne 0$, then 
$$
X_{(0, a_{10, 11, 17})}^{' 2}:=\{1, x_{5, 6, 12}(a_{10, 11, 17}/(a_{16}a_{21}a_{22}))\} \ \ \text{and} \ \ 
Y_{(0, a_{10, 11, 17})}^{' 2}:=\{1, x_4(a_{16}a_{21}a_{22}/(a_{10, 11, 17}^2))\},$$
and 
$V^{2}_{(0, a_{10, 11, 17})}$ 
is abelian. This gives the family $\cF_3^{3, p = 2}$ in Table \ref{tab:coresD6}. 

Finally, if $a_{7, 18, 19}:=c_{7, 18, 19}\ne 0$ and $
a_{10, 11, 17}:=c_{10, 11, 17} \ne 0$, then 
we have that $X_{(a_{7, 18, 19}, a_{10, 11, 17})}^{' 2}=X_{(a_{7, 18, 19}, 0)}^{' 2}X_{(0, a_{10, 11, 17})}^{' 2}$ and $Y_{(a_{7, 18, 19}, a_{10, 11, 17})}^{' 2}=Y_{(a_{7, 18, 19}, 0)}^{' 2}Y_{(0, a_{10, 11, 17})}^{' 2}$, and 
$V^{2}_{(a_{7, 18, 19}, a_{10, 11, 17})}$ 
is abelian. This 
yields the family $\cF_3^{4, p = 2}$ in Table \ref{tab:coresD6}. 

We observe that 
$$(q^6)^2|\cF_3^{1, p = 2}|+(q^6/2)^2|\cF_3^{2, p = 2}|+(q^6/2)^2|\cF_3^{3, p = 2}|+(q^6/4)^2|\cF_3^{4, p = 2}|=q^{14}(q-1)^4,$$
and since $|\cS \setminus \cZ|=14$ and $|\cZ|=4$, Equation \eqref{eq:allfam} then yields 
$$\Irr(X_{\cS})_{\cZ}=\cF_3^{1, p=2} \sqcup \cF_3^{2, p=2} \sqcup\cF_3^{3, p=2} \sqcup\cF_4^{1, p=2},$$
which is our second claim. 
\end{proof}

We now study the unique core of the form $[4, 24, 43]$ in 
type $\rD_6$. In this case, 
\begin{itemize}
\item $\cS=\{ \alpha_{1},  \dots,\alpha_{24}\}$,
\item $\cZ=\{\alpha_{21}, \alpha_{22}, \alpha_{23}, \alpha_{24}\}$,
\item $\cA = \cL = \emptyset$, 
\item $\cI=\{ \alpha_1, \alpha_5, \alpha_6\}$ and $\cJ=\{\alpha_{17}, \alpha_{18}, \alpha_{19} \}$.
\end{itemize}

\begin{prop}
\label{core[4,24,43]}
The irreducible characters corresponding to the $[4, 24, 43]$-core in type $\rD_6$ are parametrized as follows:
\begin{itemize}
\item 
If $p\neq 2$, then 
$$\Irr(X_\cS)_\cZ=\cF_5^{1, p \ne 2} \sqcup \cF_5^{2, p \ne 2} \sqcup \cF_5^{3, p \ne 2},$$
where
  \begin{itemize}
    \item $\cF_5^{1, p \ne 2}$ consists of $q(q-1)^5$ characters of degree $q^9$, 
    \item $\cF_5^{2, p\ne 2}$ consists of $(q-1)^5$ characters of degree $q^9$, and  
    \item $\cF_5^{3, p\ne 2}$ consists of $q^2(q-1)^4$ characters of degree $q^8$. 
  \end{itemize}
\item 
If $p= 2$, then 
$$\Irr(X_\cS)_\cZ=\cF_5^{1, p = 2} \sqcup \cF_5^{2, p = 2} \sqcup \cF_5^{3, p = 2} \sqcup \cF_5^{4, p = 2},$$
where
  \begin{itemize}
    \item $\cF_5^{1, p = 2}$ consists of $4q(q-1)^5$ characters of degree $q^9/2$, 
    \item $\cF_5^{2, p= 2}$ consists of $q^2(q-1)^5$ characters of degree $q^8$, 
    \item $\cF_5^{3, p= 2}$ consists of $q(q-1)^4$ characters of degree $q^8$, and 
    \item $\cF_5^{4, p= 2}$ consists of $4q(q-1)^5$ characters of degree $q^8/2$.
  \end{itemize}
\end{itemize}
The labels of the characters in $\cF_5^{1, p \ne 2}$, $\dots$, $\cF_5^{3, p \ne 2}$ 
and in $\cF_5^{1, p = 2}$, $\dots$, $\cF_5^{4, p = 2}$ are collected in Table \ref{tab:coresD6}. 
\end{prop}

\begin{proof}
The form of Equation \eqref{eq:1i} is 
\begin{align*}
s_{17}(a_{21}t_5+a_{22}t_6)+
s_{18}(-a_{21}t_1-a_{23}t_6)+
s_{19}(-a_{22}t_1-a_{23}t_5)=0.
\end{align*}

Let $p \ne 2$. Then $X'=Y'=1$, and $\overline{V}=X_2X_3X_4X_7 \cdots X_{16}X_{20}Z/(\ker \lambda)$. Notice that 
$X_2 \cap [\overline{V}, \overline{V}]=X_4 \cap [\overline{V}, \overline{V}]=1$, and $[X_i, X_{20}] \ne 1$ just for $i=2, 4$. 
Then we can take $X_2X_4$ for a candidate of an arm and 
$X_{20}$ for a candidate of a leg. We have 
$$[x_{20}(s_{20}), x_2(t_2)x_4(t_4)]=x_{23}(-s_{20}t_2)x_{24}(-s_{20}t_4).$$
Hence we apply Proposition \ref{prop:newrl} with $X^{' 2}=X_{2, 4}=\{x_{2, 4}(t) \mid t \in \F_q\}$ and $Y^{' 2}=1$, reducing to 
$V^{2}=X_{2, 4}X_3X_7 \cdots X_{16}Z/(\ker \lambda)$; here, we have  
$$X_{2,4}:=\{x_{2, 4}(t) \mid t \in \F_q\} \quad \text{ where } \quad x_{2, 4}(t):=x_2(a_{24}t)x_4(-a_{23}t).$$
We have that $X_{12}X_{14}X_{15}$ is 
a subgroup of $V^{2}$, and that 
\begin{equation} \label{eq:71011}
[X_{12}X_{14}X_{15}, X_i] \ne 1  \,\,\, \Longrightarrow \,\,\, i \in \{7 ,10, 11\}
 \text{ and } X_i \cap [V^{2}, V^{2}]=1 \text{ for } i \in \{7 ,10, 11\}. 
\end{equation}
We then apply Proposition \ref{prop:newrl} with $X_{7}X_{10}X_{11}$ 
and $X_{12}X_{14}X_{15}$ as candidates for an arm and 
a leg respectively, reducing to studying the equation 
\begin{equation}\label{eq:tri2}
s_{12}(a_{21}t_{10}+a_{22}t_{11})+
s_{14}(-a_{21}t_7-a_{24}t_{11})+
s_{15}(-a_{22}t_7-a_{24}t_{10})=0.
\end{equation}
As $p \ne 2$, we have that $X^{' 3}=Y^{' 3}=1$. 
We reduce to $V^{3}=X_{2, 4}X_3X_8X_9X_{13}X_{16}Z/(\ker \lambda)$. 

We observe that in $V^{3}$ we have that if $i=8, 9$, then 
$[X_{2, 4}, X_i]=X_{13}$ and $[X_k, X_i]\ne 1$ 
just for $k=16$, 
and that 
$X_{2, 4} \cap [V^{3}, V^{3}]=1=X_{16} \cap [V^{3}, V^{3}]$. 
Moreover, we notice that $X_{13}$ is central in $V^{3}$; we 
extend $\lambda$ to $\lambda^{c_{13}}$ in the usual way for every $c_{13} \in \F_q$. 
If $a_{13}:=c_{13} \ne 0$, we apply Proposition \ref{prop:newrl} with $X_{2, 4}X_{16}$ as a candidate for an arm, and 
$X_8X_9$ as a candidate for a leg. We have that 
\begin{equation}\label{eq:2416}
\lambda([x_{2, 4}(t)x_{16}(t_{16}), x_8(s_8)x_9(s_9)])=\phi(a_{13}t(a_{24}s_9+a_{23}s_8)+t_{16}(a_{24}s_9-a_{23}s_8)).
\end{equation}
We get that $X_{(a_{13})}^{' 4}=Y_{(a_{13})}^{' 4}=1$, and 
$V_{(a_{13})}^{4}=X_3X_{13}Z/(\ker \lambda)$ is abelian. We obtain the 
family $\cF_5^{1, p \ne 2}$ in Table \ref{tab:coresD6}.

Let us now assume that $c_{13}=0$. We examine $V_{(0)}^{3}$, and we notice 
that in this case $[X_{2, 4}, X_i]=1$ if $i=8, 9$. Hence we apply Proposition \ref{prop:newrl} with 
$X_8X_9$ as candidate for a leg, and $X_{16}$ as candidate for an arm. We get the expression as in Equation \eqref{eq:2416} 
by replacing $a_{13}$ with $0$. We obtain $X_{(0)}^{' 4}=X_{8, 9}:=\{x_{8, 9}(t) \mid t \in \F_q\}$ 
and $Y_{(0)}^{' 4}=1$. Here, we have $x_{8, 9}(t):=x_8(a_{24}t)x_9(a_{23}t)$ for every $t \in \F_q$. 
Notice that $X_{8, 9}$ is central in $V_{(0)}^{' 4}=X_{2, 4}X_3X_{8, 9}Z/(\ker \lambda)$; 
we denote by $\lambda''=\lambda'^{c_{8, 9}}$ the usual extension of $\lambda'$ to $X_{8, 9}$ 
for every $c_{8, 9} \in \F_q$. 

If $a_{8, 9}:=c_{8, 9} \ne 0$, then we have 
$$\lambda([x_{2, 4}(s), x_3(t)])=\lambda(x_{8, 9}(st))=\phi(a_{8, 9}st).$$
Proposition \ref{prop:newrl} applies again, with arm $X_{2, 4}$ and leg $X_3$, and 
we reduce to $V_{(0, a_{8, 9})}^{5}=X_{8, 9}Z/(\ker \lambda')$. 
We get the family $\cF_5^{2, p \ne 2}$ in Table \ref{tab:coresD6}. 

Finally, if 
$a_{8, 9}:=c_{8, 9} \ne 0$ then $V_{(0, 0)}^{5}:=X_{2, 4}X_3Z/(\ker \lambda')$
is abelian; this gives the family $\cF_5^{3, p \ne 2}$ in Table \ref{tab:coresD6}. 

As done at the end of Proposition \ref{core[4,18,18]}, the claim in the case $p \ne 2$ follows by a counting argument and Equation \eqref{eq:allfam}. 

Let us now assume $p=2$. In this case, we have 
$$X':=\{x_{1, 5, 6}(t) \mid t \in \F_q\} \qquad \text{and} \qquad Y':=\{x_{17, 18, 19}(s) \mid  s \in \F_q\},$$
where for every $s, t \in \F_q$, 
$$x_{1, 5, 6}(t):=x_{1}(a_{23}t)x_{5}(a_{22}t)x_{6}(a_{21}t) \qquad \text{and} \qquad 
x_{17, 18, 19}(s):=x_{17}(a_{23}s)x_{18}(a_{22}s)x_{19}(a_{21}s),$$
and $\overline{V}=X_2X_3X_4X_7 \cdots X_{16}X_{20}X'Y'Z/(\ker \lambda)$. In a similar way to the case $p \ne 2$ after computing $X'$ and $Y'$, we notice that we can apply 
Proposition \ref{prop:newrl} with $X^{' 2}=X_{2, 4}$ and $Y^{' 2}=1$. We reduce to 
$V^{2}=X_{2, 4}X_3X_7 \cdots X_{16}X'Y'Z/(\ker \lambda)$. We notice that $Y'$ is central in $V^{2}$; let us denote by 
$\lambda':=\lambda^{c_{17, 18, 19}}$ the usual extension of $\lambda$.

Suppose $a_{17, 18, 19}:=c_{17, 18, 19} \ne 0$. In the group $V$, we have 
$$
[X_i, X_{13}] \ne 1 \Rightarrow i \in \{1, 5, 6\}
\quad \text{ and } \quad X_i \cap [V, V]=1 \text{ for } i \in \{1, 5, 6\}.$$
We can then apply Proposition \ref{prop:newrl} with $X_{1, 5, 6}$ as a candidate for an arm, and $X_{13}$ 
as a candidate for a leg. In $V^{2}$, we have 
$$[x_{13}(s_{13}), x_{1, 5, 6}(t)]=x_{17, 18, 19}(s_{13}t)x_{21}(a_{22}a_{23}s_{13}t^2)
x_{22}(a_{21}a_{23}s_{13}t^2)x_{23}(a_{21}a_{22}s_{13}t^2),$$
hence applying $\lambda'$ we obtain the following equation,  
\begin{equation}\label{eq:15613}
\phi(a_{17, 18, 19}s_{13}t+a_{21}a_{22}a_{23}s_{13}t^2)=1.
\end{equation}
We have that 
$$X_{(a_{17, 18, 19})}^{' 3}=\{1, x_{1, 5, 6}(a_{17, 18, 19}/(a_{21}a_{22}a_{23}))\} \text{ and } 
Y_{(a_{17, 18, 19})}^{' 3}:=\{1, x_{13}(a_{21}a_{22}a_{23}/(c_{17, 18, 19}^2))\},$$
and $V_{(a_{17, 18, 19})}^{3}=X_{2, 4}X_3X_7 \cdots X_{12} X_{14}X_{15} X_{16}X_{(a_{17, 18, 19})}^{' 3}Y_{(a_{17, 18, 19})}^{' 3}Y'Z/(\ker \lambda')$. 
In this subquotient, we have that $[X_{2, 4}, X_{12}X_{14}X_{15}] \cap X_{17, 18, 19} \ne 0$, 
that $X_{2, 4} \cap [V_{(a_{17, 18, 19})}^{3}, V_{(a_{17, 18, 19})}^{3}]$=1, and that Equation \eqref{eq:71011} holds. 
Moreover, recall that in $V$ we have that if $k \in \{7, 10, 11\}$, then $[X_i, X_j] \cap X_k \ne 1$ implies $i \in \{2, 4\}$ or $j \in \{2, 4\}$. 
We can then take $X_{2, 4}X_7X_{10}X_{11}$ and $X_{12}X_{14}X_{15}$ as candidates for an arm and a 
leg respectively. We get the equation 
$$\lambda([x_{12}(s_{12})x_{14}(s_{14})x_{15}(s_{15}), 
x_{2, 4}(t)x_7(t_7)x_{10}(t_{10})x_{11}(t_{11})])=\lambda(x_{17}(a_{23}s_{12}t_1)x_{18}(a_{24}s_{14}t_1)\cdot$$
$$\cdot x_{19}(a_{24}s_{15}t_1)) \phi(s_{12}(a_{21}t_{10}+a_{22}t_{11})+s_{14}(a_{21}t_{7}+a_{24}t_{11})+s_{15}(a_{22}t_{7}+a_{24}t_{10}))=1.$$
We get that $X_{(a_{17, 18, 19})}^{' 4}=X_{7, 10, 11}:=\{x_{7, 10, 11}(t) \mid t \in \F_q\}$ and $Y_{(a_{17, 18, 19})}^{' 4}=1$, where for every 
$t \in \F_q$ we have $x_{7, 10, 11}(t):=x_7(a_{24}t)x_{10}(a_{22}t)x_{11}(a_{21}t)$, and 
$$V_{(a_{17, 18, 19})}^{4}=X_3X_{7, 10, 11}X_8X_9 X_{16}X_{(a_{17, 18, 19})}^{' 3}Y_{(a_{17, 18, 19})}^{' 3}Y'Z/(\ker \lambda').$$

Notice that $X_{(a_{17, 18, 19})}^{' 4} X_{7, 10, 11}$ 
is a subgroup of $V_{a_{17, 18, 19}}$, and that $X_3$ is there a direct product factor. 
Observe then that $[X_8, X_9]=[X_{16}, X_{7, 10, 11}]=1$, and that 
$$\lambda([x_8(s_8)x_9(s_9), x_{7, 10, 11}(t)x_{16}(t_{16})])=
\lambda(x_{17}(a_{24}s_9t)x_{18}(a_{22}s_8t)x_{19}(a_{21}s_8t))
\phi(a_{23}s_8t_{16}+a_{24}s_9t_{16}).$$
As $a_{17, 18, 19} \ne 0$, applying Proposition \ref{prop:newrl} with arm $X_{7, 10, 11}X_{16}$ and leg 
$X_{7, 10, 11}X_{16}$ yields $X_{(a_{17, 18, 19})}^{' 5}=Y_{(a_{17, 18, 19})}^{' 5}=1$, and the subquotient 
$V_{(a_{17, 18, 19})}^{5}=X_3X_{(a_{17, 18, 19})}^{' 3}Y_{(a_{17, 18, 19})}^{' 3}Y'Z/(\ker \lambda')$ of $V$ is abelian. 
We obtain the family $\cF_5^{1, p = 2}$ in Table \ref{tab:coresD6}. 
 
Let us now assume $c_{17, 18, 19} = 0$. As done for $c_{17, 18, 19} \ne 0$, we take $X_{1, 5, 6}$ 
and $X_{13}$ as candidates for an arm and a leg respectively, but as we have no $a_{17, 18, 19}$ term in 
Equation \eqref{eq:15613} we now get $X_{(0)}^{' 3}=Y_{(0)}^{' 3}=1$ and 
$V_{(0)}^{3}=X_{2, 4}X_3X_7 \cdots X_{12} X_{14}X_{15} X_{16}Z/(\ker \lambda')$. Notice that 
in this subquotient we have $[X_{2, 4}, X_{j}]=1$ for $j=8, 9, 16$, and that $[X_{16}, X_i] \ne 1$ implies 
$i \in \{8, 9\}$. We can apply Proposition \ref{prop:newrl} with $X_{16}$ as a candidate for an arm, and 
$X_8X_9$ as a candidate for a leg. We have that 
$$\lambda([x_8(s_8)x_9(s_9), x_{16}(t_{16})])=\lambda(x_{23}(s_8t_{16})x_{24}(s_9t_{16}))=
\phi(t_{16}(a_{23}s_8+a_{24}s_9)).$$
We then get $X_{(0)}^{' 4}=1$ and $Y_{(0)}^{' 4}=\{x_{8, 9}(s) \mid s \in \F_q\}$, where 
$x_{8, 9}(s)=x_8(a_{24}s)x_9(a_{23}s)$ for every $s \in \F_q$, and 
$$V_{(0)}^{4}=X_{2, 4}X_3X_7 X_{8, 9}X_{10}X_{11} X_{12} X_{14}X_{15} Z/(\ker \lambda').$$ 
Now we observe that \eqref{eq:71011} holds with $V_{(0)}^{4}$ in place of $V^{2}$, as 
$X_{20}$ and $X_{17}X_{18}X_{19}$ are contained in $\ker\lambda'$. 
We take $X_7X_{10}X_{11}$ as 
a candidate for an arm and $X_{12}X_{14}X_{15}$ 
as a candidate for a leg. Equation \eqref{eq:tri2} yields in this case 
$$X_{(0)}^{' 5}:=X_{7, 10, 11}=\{x_{7, 10, 11}(t) \mid t \in \F_q\} \quad \text{ and } \quad 
Y_{(0)}^{' 5}:=X_{12, 14, 15}=\{x_{12, 14, 15}(s) \mid s \in \F_q\},$$ 
and $V_{(0)}^{5}=X_{2, 4}X_3 X_{8, 9} X_{7, 10, 11}X_{12, 14, 15}Z/(\ker \lambda')$. Here, for $s, t \in \F_q$ we have 
$$x_{7, 10, 11}(t)=x_7(a_{24}t)x_{10}(a_{22}t)x_{11}(a_{21}t) \quad \text{ and } \quad 
x_{12, 14, 15}(s)=x_{12}(a_{24}s)x_{14}(a_{22}s)x_{15}(a_{21}s).$$ 

Finally, we observe that $X_{8, 9}$ and $X_{12, 14, 15}$ are central in $V_{(0)}^{5}$; we extend $\lambda'$ to 
$\lambda'':=\lambda'^{c_{8, 9}, c_{12, 14, 15}}$ in the usual way. Observe that $[X_{2, 4}, X_{7, 10, 11}]=1$. We can then 
take $X_{2, 4}X_{7, 10, 11}$ and $X_3$ as candidates for an arm and a leg respectively. We study
$$\lambda([x_3(s_3), x_{2, 4}(t_1)x_{7, 10, 11}(t_2)])
=\phi(s_3(c_{8, 9}t_1+c_{12, 14, 15}t_2+a_{21}a_{22}a_{24}t_2^2))=1.$$ 

If $a_{8, 9}:=c_{8, 9}\ne 0$ and $b_{12, 14, 15}:=c_{12, 14, 15}$ is arbitrary in $\F_q$, we have that 
$$X_{(0, a_{8, 9}, b_{12, 14, 15})}^{' 6}=\{x_{2, 4}((b_{12, 14, 15}t+a_{21}a_{22}a_{24}t^2)/(a_{8, 9}^2))
x_{7, 10, 11}(t) \mid t \in \F_q\}  \text{ and }  Y_{(0, a_{8, 9}, b_{12, 14, 15})}^{' 6}=1,$$
and $V_{(0, a_{8, 9}, b_{12, 14, 15})}^{6}=X_{(0, a_{8, 9}, b_{12, 14, 15})}^{' 6} X_{8, 9} Y_{(0)}^{' 5}Z/(\ker \lambda'')$ is 
abelian; this gives the family $\cF_5^{2, p = 2}$ in Table \ref{tab:coresD6}. 

If 
$c_{8, 9}=0$ and $a_{12, 14, 15}:=c_{12, 14, 15}\ne 0$, we have that 
$$X_{(0, 0, a_{12, 14, 15})}^{' 6}=X_{2,4}\{1, x_{7, 10, 11}(c_{12, 14, 15}/(a_{21}a_{22}a_{24}))\}$$ 
and 
$$Y_{(0, 0, a_{12, 14, 15})}^{' 6}=\{1, x_3(a_{21}a_{22}a_{24}/(c_{12, 14, 15}^2))\},$$
and $V_{(0, 0, a_{12, 14, 15})}^{6}=X_{(0, 0, a_{12, 14, 15})}^{' 6}Y_{(0, 0, a_{12, 14, 15})}^{' 6} Y_{(0)}^{' 5}Z/(\ker \lambda'')$ 
is abelian; we obtain the family $\cF_5^{4, p = 2}$ in Table \ref{tab:coresD6}. 

If $c_{8, 9}=c_{12, 14, 15}=0$, we have that 
$$X_{(0, 0, 0)}^{' 6}=X_{2,4} \quad \text{ and } \quad Y_{(0, 0, a_{12, 14, 15})}^{' 6}=1,$$
and $V_{(0, 0, 0)}^{6}=X_{2, 4} Z/(\ker \lambda'')$ 
is abelian. This yields the family $\cF_5^{3, p = 2}$ in Table \ref{tab:coresD6}.

As done for the case $p \ne 2$, we check that  
$$\Irr(X_\cS)_\cZ=\cF_5^{1, p = 2} \sqcup \cF_5^{2, p = 2} \sqcup \cF_5^{3, p = 2} \sqcup \cF_5^{4, p = 2}$$
by apply the counting argument and Equation \eqref{eq:allfam}. This concludes our analysis. 
\end{proof}

Finally, we study the unique core of the form $[5,20,25]$ in type $\rE_6$. In this case,
\begin{itemize}
\item $\cS=\{\alpha_{1}, \alpha_{2}, \alpha_{3}, \alpha_{
5}, \alpha_{6}, \alpha_{7}, \alpha_{8}, \alpha_{9}, \alpha_{10}, \alpha_{11}, \alpha_{12}, \alpha_{13}, \alpha_{14}, \alpha_{15}, \alpha_{16}, \alpha_{17}, \alpha_{18}, \alpha_{20}, \alpha_{21}, \alpha_{24}\}$,
\item $\cZ=\{\alpha_{17}, \alpha_{18}, \alpha_{20}, \alpha_{21}, \alpha_{24}\}$,
\item $\cA=\{\alpha_{4}\}$ and 
$\cL=\{\alpha_{19}\}$, 
\item $\cI=\{\alpha_{1}, \alpha_{6}, \alpha_{8}, \alpha_{9}, \alpha_{10}\}$ and $\cJ=\{\alpha_{7}, \alpha_{11}, \alpha_{13}, \alpha_{14}, \alpha_{15}\}$.
\end{itemize}

\begin{prop}
\label{core[5,20,25]}
The irreducible characters corresponding to the $[5, 20, 25]$-core in type $\rE_6$ are parametrized as follows: 
\begin{itemize}
\item 
If $p\neq 3$, then $\Irr(X_\cS)_\cZ=\cF_8^{p \ne 3}$ consists of $q(q-1)^5$ characters of degree $q^7$. 
\item If $p= 3$, then 
$$\Irr(X_\cS)_\cZ=\cF_8^{1, p = 3} \sqcup \cF_8^{2, p = 3},$$
where
  \begin{itemize}
    \item $\cF_8^{1, p = 3}$ consists of $(q-1)^6$ characters of degree $q^7$, and 
    \item $\cF_8^{2, p = 3}$ consists of $q^2(q-1)^5$ characters of degree $q^6$. 
  \end{itemize}
\end{itemize}
The labels of the characters in $\cF_8^{p \ne 3}$ and in $\cF_8^{1, p = 3}$, $\cF_8^{2, p = 3}$ are collected in Table \ref{tab:coresE6}. 
\end{prop}

\begin{proof}
The form of Equation \eqref{eq:1i} is 
\begin{align*}
s_7(a_{17}t&_8+a_{18}t_{10})+
s_{11}(-a_{20}t_8-a_{21}t_9)+
s_{13}(-a_{17}t_1+a_{24}t_{10})+\\
&+s_{14}(a_{20}t_6+a_{24}t_9)+
s_{15}(-a_{18}t_1+a_{21}t_6+a_{24}t_8)=0.
\end{align*}
Let $p \ne 3$. Then $X'=Y'=1$, and 
$\overline{V}= X_2X_3X_5X_{12}X_{16}Z/(\ker \lambda)$. Observe that in $\overline{V}$ the pairs of root subgroups 
that give nontrivial commutator brackets are exactly the following, 
\begin{equation}\label{eq:com520}
[X_2, X_{12}]=X_{17},\quad [X_2, X_{16}]=X_{20},\quad 
[X_3, X_{16}]=X_{21},\quad [X_5, X_{12}]=X_{18}.
\end{equation}
We apply Proposition \ref{prop:newrl} with $X_2X_3X_5$ as a candidate for 
an arm, and $X_{12}X_{16}$ as a candidate for a leg. We have  
$$[x_2(t_2)x_3(t_3)x_5(t_5), x_{12}(s_{12})x_{16}(s_{16})]
=x_{17}(s_{12}t_2)x_{18}(-s_{12}t_5)x_{20}
(s_{16}t_2)x_{21}(s_{16}t_3),$$
hence 
$$\lambda([x_2(t_2)x_3(t_3)x_5(t_5), x_{12}(s_{12})])
=\phi(s_{12}(a_{17}t_2-a_{18}t_5)+s_{16}(a_{20}t_2+a_{21}t_3)).$$
We get $X^{' 2}=\{x_{2, 3, 5}(t) \mid t \in \F_q\}$ and $Y^{' 2}=1$, 
where 
$$x_{2, 3, 5}(t):=x_2(a_{18}a_{21}t)x_3(-a_{18}a_{20}t)x_5(a_{17}a_{21}t)$$ for every $t \in \F_q$. As $V^2=X''Z/(\ker \lambda)$ is abelian, we get 
the family $\cF_8^{p \ne 3}$ in Table \ref{tab:coresE6}.

Let us now assume that $p=3$. Then we have 
$$X':=\{x_{1, 6, 8, 9, 10}(t) \mid t \in \F_q\} \qquad \text{and} \qquad Y':=\{x_{7, 11, 13, 14, 15}(s) \mid s \in \F_q\},$$
where for every $s, t \in \F_q$, 
$$x_{1, 6, 8, 9, 10}(t):=x_{1}(a_{21}a_{24}t)x_{6}(-a_{18}a_{24}t)x_{8}(-a_{18}a_{21}t)x_{9}(a_{18}a_{20}t)x_{10}(a_{17}a_{21}t)$$
and
$$x_{7, 11, 13, 14, 15}(s):=x_{7}(a_{20}a_{24}s)x_{11}(-a_{17}a_{24}s)x_{13}(-a_{18}a_{20}s)x_{14}(-a_{17}a_{21}s)x_{15}(a_{17}a_{20}s),$$
and $\overline{V}= X_2X_3X_5X_{12}X_{16}X'Y'Z/(\ker \lambda)$. We extend $\lambda$ to $\lambda'=\lambda^{c_{7, 11, 13, 14, 15}}$, $c_{7, 11, 13, 14, 15} \in \F_q$. 

Notice that $X'$ is a subgroup of $\overline{V}$. 
Moreover, the nontrivial commutator relations in $\overline{V}$ are 
as in Equation \eqref{eq:com520}, plus $[X', X_i] \ne 1$ if and 
only if $i \in \{2, 3, 5\}$, in which case such a commutator 
lies inside $Y'$. In this case, Proposition \ref{prop:newrl} 
applies with $X'X_{12}X_{16}$ as a candidate for 
an arm and $X_2X_3X_5$ as a candidate for a leg. We study the equation  
\begin{align*}
\lambda(&[x_{1,6,8,9,10}(t)x_{12}(t_{12})x_{16}(t_{16}), x_2(s_2)x_3(s_3)x_5(s_5)])=
\lambda(
x_7(-a_{21}a_{24}s_3t)x_{11}(-a_{18}a_{24}s_5t))\cdot 
\\
&\cdot\lambda(x_{13}(a_{18}a_{20}s_2t-a_{18}a_{21}s_3t)
x_{14}(a_{18}a_{21}s_5t+a_{17}a_{21}s_2t)
x_{15}(-a_{18}a_{20}s_5t+a_{17}a_{21}s_3t))\cdot \\
&\cdot \phi(s_2(a_{17}t_{12}+a_{20}t_{16}+
a_{17}a_{18}a_{20}a_{21}a_{24}t^2)
+s_3(a_{21}t_{16}-a_{17}a_{18}a_{21}^2a_{24}t^2))\cdot\\
&\cdot \phi(s_5(-a_{18}t_{12}-a_{18}^2a_{20}a_{21}a_{24}t^2))=1.
\end{align*}
If $a_{7, 11,  13, 14, 15}:=c_{7, 11,  13, 14, 15}\ne 0$, then we have that $X_{(a_{7, 11,  13, 14, 15})}^{' 2}
=Y_{(a_{7, 11,  13, 14, 15})}^{' 2}=1$, and $V_{(a_{7, 11,  13, 14, 15})}^2= Y'Z/(\ker \lambda')$ is abelian. 
We get the family $\cF_8^{1, p = 3}$ in Table \ref{tab:coresE6}. 

If $c_{7, 11,  13, 14, 15}= 0$, 
then we have 
$$X_{(0)}^{',2}=\{x_{1,6,8,9,10}(t)x_{12}(a_{18}a_{20}a_{21}a_{24}t^2)x_{16}(a_{17}a_{18}a_{21}a_{24}t^2) \mid t \in \F_q\},$$
$$Y_{(0)}^{',2}=\{x_2(a_{18}a_{21}s)x_3(-a_{18}a_{20}s)x_5(a_{17}a_{21}s) \mid s \in \F_q\},$$
and $V_{(0)}^2= X_{(0)}^{',2}Y_{(0)}^{',2}Z/(\ker \lambda')$ is 
abelian. This yields the family $\cF_8^{2, p = 3}$ in Table \ref{tab:coresE6}. 

The claim now follows by Equation \eqref{eq:allfam} as done in Propositions \ref{core[4,18,18]} and \ref{core[4,24,43]}. 
\end{proof}
 
\begin{table}[h]
\footnotesize
\begin{tabular}{|l|l|}
\hline
$D$ & $k(\mathrm{UD}_6(p^d), D), p \ge 3$  \\
\hhline{|=|=|}
 $1$ & $v^6 +6v^5 +15v^4 +20v^3 +15v^2 +6v+1$   \\
\hline
 $q$ & $v^7 +9v^6 +31v^5 +54v^4 +51v^3 +25v^2+5v$   \\
\hline
$q^2$ & $v^8 +9v^7 +38v^6 +89v^5 +119v^4 +89v^3 +34v^2 +5v$   \\
\hline
$q^3$ & $v^8 +15v^7 +72v^6 +165v^5 +201v^4 +130v^3 +40v^2 +4v$   \\
\hline
$q^4$ & $3v^8 +31v^7 +124v^6 +246v^5 +260v^4 +145v^3 +39v^2 +4v$ \\
\hline
$q^5$ & $v^{10} +10v^9 +46v^8 +135v^7 +280v^6 +393v^5 +339v^4 +163v^3 +36v^2 +2v$ \\
\hline
$q^6$ & $2v^9 +18v^8 +77v^7 +200v^6 +317v^5 +288v^4 +138v^3 +30v^2 +2v$  \\
\hline
$q^7$ & $5v^8 +43v^7 +154v^6 +282v^5 +270v^4 +128v^3 +25v^2 +v $   \\
\hline
$q^8$ & $3v^8 +31v^7 +122v^6 +227v^5 +208v^4 +89v^3 +15v^2 +v$   \\
\hline
$q^9$ & $v^9 +9v^8 +41v^7 +113v^6 +181v^5 +152v^4 +61v^3 +8v^2$  \\
\hline
$q^{10}$ & $v^8 +8v^7 +31v^6 +62v^5 +61v^4 +27v^3 +5v^2$   \\
\hline
$q^{11}$ & $2v^7 +12v^6 +29v^5 +32v^4 +15v^3 +2v^2$   \\
\hline
$q^{12}$ & $v^6 +4v^5 +6v^4 +4v^3 +v^2$   \\
\hline
\hline 
\multicolumn{2}{|c|}{\begin{footnotesize}$k(\mathrm{UD}_6(q))= v^{10} + 13v^9 + 87v^8 + 393v^7 + 1157v^6 + 2032v^5 + 2005v^4 + 1060v^3 + 275v^2+ 30v + 1$\end{footnotesize}}\\
\hline
\end{tabular}
\caption{The numbers of irreducible characters of $\mathrm{UD}_6(q)$ of fixed degree for $q=p^d$, $p \ge 3$, where $v=q-1$.}
\label{tab:fam3D6}
\end{table}
 
\begin{table}[!h]
\footnotesize
\begin{tabular}{|l|l|}
\hline
$D$ & $k(\mathrm{UD}_6(2^d), D)$  \\
\hhline{|=|=}
 $1$ & $v^6 + 6v^5 + 15v^4 + 20v^3 + 15v^2 + 6v + 1  $   \\
\hline
 $q$ & $v^7 + 9v^6 + 31v^5 + 54v^4 + 51v^3 + 25v^2 + 5v $   \\
\hline
$q^2$ & $ v^8 + 9v^7 + 38v^6 + 89v^5 + 119v^4 + 89v^3 + 34v^2 + 5v $   \\
\hline
$q^3/2$ & $ 4v^6 + 8v^5 + 4v^4  $   \\
\hline
$q^3$ & $v^8 + 15v^7 + 71v^6 + 163v^5 + 200v^4 + 130v^3 + 40v^2 + 4v $   \\
\hline
$q^4/2$ & $  4v^7 + 16v^6 + 16v^5 + 4v^4 $   \\
\hline
$q^4$ & $4v^8 + 35v^7 + 128v^6 + 247v^5 + 260v^4 + 145v^3 + 39v^2 + 4v $ \\
\hline
$q^5/2$ & $4v^7 + 16v^6 + 20v^5 + 8v^4   $   \\
\hline
$q^5$ & $v^{10} + 10v^9 + 46v^8 + 135v^7 + 278v^6 + 388v^5 + 337v^4 + 163v^3 + 36v^2 + 2v $ \\
\hline
$q^6/2$ & $ 8v^7 + 28v^6 + 28v^5 + 8v^4  $   \\
\hline
$q^6$ & $2v^9 + 18v^8 + 76v^7 + 196v^6 + 312v^5 + 286v^4 + 138v^3 + 30v^2 + 2v$  \\
\hline
$q^7/2$ & $4v^7 + 24v^6 + 32v^5 + 12v^4   $   \\
\hline
$q^7$ & $6v^8 + 47v^7 + 157v^6 + 280v^5 + 268v^4 + 128v^3 + 25v^2 + v  $   \\
\hline
$q^8/2$ & $ 8v^7 + 36v^6 + 36v^5 + 8v^4  $   \\
\hline
$q^8$ & $4v^8 + 35v^7 + 122v^6 + 221v^5 + 205v^4 + 89v^3 + 15v^2 + v$   \\
\hline
$q^9/2$ & $ 12v^7 + 40v^6 + 36v^5 + 8v^4  $   \\
\hline
$q^9$ & $v^9 + 9v^8 + 38v^7 + 102v^6 + 168v^5 + 149v^4 + 61v^3 + 8v^2$  \\
\hline
$q^{10}/4$ & $ 16v^6 $   \\
\hline
$q^{10}/2$ & $ 8v^7 + 20v^6 + 28v^5 + 4v^4  $   \\
\hline
$q^{10}$ & $v^8 + 6v^7 + 25v^6 + 55v^5 + 60v^4 + 27v^3 + 5v^2 $   \\
\hline
$q^{11}/2$ & $ 8v^6 + 12v^5 + 4v^4  $   \\
\hline
$q^{11}$ & $2v^7 + 10v^6 + 26v^5 + 31v^4 + 15v^3 + 2v^2 $   \\
\hline
$q^{12}$ & $v^6 + 4v^5 + 6v^4 + 4v^3 + v^2 $   \\
\hline
\hline 
\multicolumn{2}{|c|}{\begin{footnotesize}$k(\mathrm{UD}_6(q))= v^{10} + 13v^9 + 90v^8 + 447v^7 + 1346v^6 + 2206v^5 + 2050v^4 + 1060v^
3 + 275v^2 + 30v + 1$\end{footnotesize}}\\
\hline
\end{tabular}
\caption{The numbers of irreducible characters of $\mathrm{UD}_6(q)$ of fixed degree for $q=2^d$, where $v=q-1$.}
\label{tab:fam2D6}
\end{table}
 
\begin{table}[!h]
\footnotesize
\begin{tabular}{|l|l|}
\hline
$D$ & $k(\mathrm{UE}_6(p^d), D), p \ge 5$  \\
\hhline{|=|=|}
 $1$ & $v^6 +6v^5 +15v^4 +20v^3 +15v^2 +6v+1$   \\
\hline
 $q$ & $v^7 +9v^6 +31v^5 +54v^4 +51v^3 +25v^2 +5v$   \\
\hline
$q^2$ & $5v^7 +34v^6 +93v^5 +130v^4 +97v^3 +36v^2 +5v$   \\
\hline
$q^3$ & $v^9 +9v^8 +42v^7 +123v^6 +223v^5 +240v^4 +145v^3 +44v^2 +5v$   \\
\hline
$q^4$ & $5v^8 +42v^7 +155v^6 +300v^5 +316v^4 +176v^3 +46v^2 +4v$ \\
\hline
$q^5$ & $2v^9 +23v^8 +118v^7 +327v^6 +518v^5 +462v^4 +219v^3 +48v^2 +3v$ \\
\hline
$q^6$ & $14v^8 +113v^7 +367v^6 +602v^5 +523v^4 +231v^3 +45v^2 +3v$  \\
\hline
$q^7$ & $v^{11} +11v^{10} +57v^9 +186v^8 +433v^7 +730v^6 +826v^5 +560v^4 +204v^3 +36v^2 +2v $   \\
\hline
$q^8$ & $v^{10} +10v^9 +51v^8 +173v^7 +396v^6 +558v^5 +444v^4 +183v^3 +31v^2 +v$   \\
\hline
$q^9$ & $3v^9 +30v^8 +144v^7 +385v^6 +575v^5 +455v^4 +177v^3 +28v^2 +v$  \\
\hline
$q^{10}$ & $12v^8 +95v^7 +304v^6 +480v^5 +375v^4 +131v^3 +16v^2 +v$   \\
\hline
$q^{11}$ & $2v^9 +21v^8 +97v^7 +243v^6 +334v^5 +233v^4 +71v^3 +10v^2$   \\
\hline
$q^{12}$ & $2v^8 +20v^7 +76v^6 +139v^5 +124v^4 +49v^3 +6v^2$   \\
\hline
$q^{13}$ & $3v^7 +24v^6 +63v^5 +68v^4 +28v^3 +3v^2$   \\
\hline
$q^{14}$ & $4v^6 +19v^5 +27v^4 +12v^3 +v^2$   \\
\hline
$q^{15}$ & $3v^5+8v^4+5v^3$   \\
\hline
$q^{16}$ & $v^4+v^3$   \\
\hline
\hline 
\multicolumn{2}{|c|}{\begin{footnotesize}$k(\mathrm{UE}_6(q))= v^{11} + 12v^{10} + 75v^9 + 353v^8 + 1286v^7 + 3178v^6 + 4770v^5 + 4035v^4 + 1800v^3 + 390v^2 + 36v + 1$\end{footnotesize}}\\
\hline
\end{tabular}
\caption{The numbers of irreducible characters of $\mathrm{UE}_6(q)$ of fixed degree for $q=p^d$, $p \ge 5$, where $v=q-1$.}
\label{tab:fam5E6}
\end{table}

\begin{table}[!h]
\footnotesize
\begin{tabular}{|l|l|}
\hline
$D$ & $k(\mathrm{UE}_6(3^d), D)$  \\
\hhline{|=|=|}
 $1$ & $v^6+6v^5+15v^4+20v^3+15v^2+6v+1$   \\
\hline
 $q$ & $v^7+9v^6+31v^5+54v^4+51v^3+25v^2+5v$   \\
\hline
$q^2$ & $5v^7+34v^6+93v^5+130v^4+97v^3+36v^2+5v$   \\
\hline
$q^3$ & $v^9+9v^8+42v^7+123v^6+223v^5+240v^4+145v^3+44v^2+5v$   \\
\hline
$q^4$ & $5v^8+42v^7+155v^6+300v^5+316v^4+176v^3+46v^2+4v$ \\
\hline
$q^5$ & $2v^9+23v^8+118v^7+327v^6+518v^5+462v^4+219v^3+48v^2+3v$ \\
\hline
$q^6$ & $14v^8+113v^7+367v^6+602v^5+523v^4+231v^3+45v^2+3v$  \\
\hline
$q^7/3$ & $9v^6/2$   \\
\hline
$q^7$ & $v^{11}+11v^{10}+57v^9+186v^8+434v^7+1463v^6/2+827v^5+560v^4+204v^3+36v^2+2v $   \\
\hline
$q^8$ & $v^{10}+10v^9+52v^8+178v^7+403v^6+560v^5+444v^4+183v^3+31v^2+v$   \\
\hline
$q^9$ & $3v^9+30v^8+144v^7+384v^6+572v^5+455v^4+177v^3+28v^2+v$  \\
\hline
$q^{10}$ & $12v^8+95v^7+304v^6+480v^5+375v^4+131v^3+16v^2+v$   \\
\hline
$q^{11}$ & $2v^9+21v^8+97v^7+243v^6+334v^5+233v^4+71v^3+10v^2$   \\ 
\hline
$q^{12}$ & $2v^8+20v^7+76v^6+139v^5+124v^4+49v^3+6v^2$   \\
\hline
$q^{13}$ & $3v^7+24v^6+63v^5+68v^4+28v^3+3v^2$   \\
\hline
$q^{14}$ & $4v^6+19v^5+27v^4+12v^3+v^2$   \\
\hline
$q^{15}$ & $3v^5+8v^4+5v^3$   \\
\hline
$q^{16}$ & $v^4+v^3$   \\
\hline
\hline 
\multicolumn{2}{|c|}{\begin{footnotesize}$k(\mathrm{UE}_6(q))= v^{11} + 12v^{10} + 75v^9 + 354v^8 + 1292v^7 + 3190v^6 + 4770v^5 + 4035v^4 + 1800v^3 + 390v^2 + 36v + 1$\end{footnotesize}}\\
\hline
\end{tabular}
\caption{The numbers of irreducible characters of $\mathrm{UE}_6(q)$ of fixed degree for $q=3^d$, where $v=q-1$.}
\label{tab:fam3E6}
\end{table}

\begin{table}[!h]
\footnotesize
\begin{center}
\begin{tabular}{|l|l|}
\hline
$D$ & $k(\mathrm{UE}_6(2^d), D)$  \\
\hhline{|=|=|}
 $1$ & $v^6+6v^5+15v^4+20v^3+15v^2+6v+1$   \\
\hline
 $q$ & $v^7+9v^6+31v^5+54v^4+51v^3+25v^2+5v$   \\
\hline
$q^2$ & $5v^7+34v^6+93v^5+130v^4+97v^3+36v^2+5v$   \\
\hline
$q^3/2$ & $4v^6+8v^5+4v^4$   \\
\hline
$q^3$ & $v^9+9v^8+42v^7+122v^6+221v^5+239v^4+145v^3+44v^2+5v$   \\
\hline
$q^4/2$ & $8v^6+16v^5+8v^4$   \\
\hline
$q^4$ & $5v^8+44v^7+159v^6+302v^5+316v^4+176v^3+46v^2+4v$ \\
\hline
$q^5/2$ & $12v^6+24v^5+12v^4$   \\
\hline
$q^5$ & $2v^9+24v^8+123v^7+333v^6+517v^5+459v^4+219v^3+48v^2+3v$ \\
\hline
$q^6/2$ & $16v^6+32v^5+16v^4$   \\
\hline
$q^6$ & $14v^8+115v^7+368v^6+597v^5+519v^4+231v^3+45v^2+3v$  \\
\hline
$q^7/2$ & $24v^7+92v^6+92v^5+20v^4$   \\
\hline
$q^7$ & $v^{11}+11v^{10}+57v^9+188v^8+437v^7+723v^6+811v^5+555v^4+204v^3+36v^2+2v$   \\
\hline
$q^8/2$ & $4v^7+28v^6+44v^5+20v^4$   \\
\hline
$q^8$ & $v^{10}+10v^9+51v^8+176v^7+399v^6+553v^5+441v^4+183v^3+31v^2+v$   \\
\hline
$q^9/2$ & $8v^7+44v^6+56v^5+20v^4$   \\
\hline
$q^9$ & $3v^9+32v^8+154v^7+398v^6+577v^5+452v^4+177v^3+28v^2+v$  \\
\hline
$q^{10}/2$ & $4v^7+28v^6+44v^5+20v^4$   \\
\hline
$q^{10}$ & $13v^8+102v^7+314v^6+479v^5+370v^4+131v^3+16v^2+v$   \\
\hline
$q^{11}/2$ & $4v^7+36v^6+56v^5+20v^4$   \\
\hline
$q^{11}$ & $2v^9+21v^8+98v^7+239v^6+320v^5+224v^4+71v^3+10v^2$   \\
\hline
$q^{12}/2$ & $12v^6+28v^5+16v^4$   \\
\hline
$q^{12}$ & $2v^8+20v^7+74v^6+132v^5+119v^4+49v^3+6v^2$   \\
\hline
$q^{13}/2$ & $8v^6+24v^5+12v^4$   \\
\hline
$q^{13}$ & $3v^7+22v^6+57v^5+64v^4+28v^3+3v^2$   \\
\hline
$q^{14}/2$ & $8v^5+8v^4$   \\
\hline
$q^{14}$ & $4v^6+17v^5+25v^4+12v^3+v^2$   \\
\hline
$q^{15}/2$ & $4v^4$   \\
\hline
$q^{15}$ & $3v^5+7v^4+5v^3$   \\
\hline
$q^{16}$ & $v^4+v^3$   \\
\hline
\hline 
\multicolumn{2}{|c|}{\begin{footnotesize}$k(\mathrm{UE}_6(q))=v^{11} + 12v^{10} + 75v^9 + 359v^8 + 1364v^7 + 3487v^6 + 5148v^5 + 4170v^4 + 1800v^3 + 390v^2 + 36v + 1$\end{footnotesize}}\\
\hline
\end{tabular}
\end{center}
\caption{The numbers of irreducible characters of $\mathrm{UE}_6(q)$ of fixed degree for $q=2^d$, where $v=q-1$.}
\label{tab:fam2E6}
\end{table}


\begin{thebibliography}{99}

\bibitem[MAGMA]{MAGMA} W. Bosma, J. Cannon and C. Playoust, \emph{The Magma algebra system. I. The user language}, J. Symbolic Comput. \textbf{24} (1997), 235--265.

\bibitem[Car]{Car}
R. W. Carter, {\em Simple groups of Lie type}, Vol. 22, John Wiley \& sons, 1989.

\bibitem[DM]{DM}  F. Digne and J. Michel, \emph{Representations of finite groups of Lie type}, Vol. 21, Cambridge University Press, 1991. 

\bibitem[DM15]{DM15} O. Dudas and G. Malle, \emph{Decomposition matrices for low-rank unitary groups}, Proc. London Math. Soc. \textbf{110} (2015), no.\ 6, 1517--1557.

\bibitem[DM16]{DM16} O. Dudas and G. Malle, \emph{Decomposition matrices for exceptional groups at $d= 4$}, J. Pure Appl. Algebra \textbf{220} (2016), no.\ 3, 1096--1121.

\bibitem[CHEVIE]{CHEVIE}
M. Geck, G. Hiss, F. L\"ubeck, G. Malle and G. Pfeiffer, {\em CHEVIE -- A system for computing
and processing generic character tables for finite groups of Lie type, Weyl groups and Hecke
algebras}, Appl. Algebra Engrg. Comm. Comput. {\bf 7} (1996), 175--210.

\bibitem[GHM96]{GHM96} M. Geck, G. Hiss, and G. Malle, \emph{Towards a classification of the irreducible representations in non-defining characteristic of a finite group of Lie type}, Math. Z. \textbf{221} (1996), 353--386. 

\bibitem[GLM17]{GLM17} S. M. Goodwin, T. Le and K. Magaard, {\em The generic character table of a Sylow $p$-subgroup of a finite Chevalley group of type $D_4$}, Comm. Algebra 45 (2017), 5158--5179.

\bibitem[GLMP16]{GLMP16} S. M. Goodwin, T. Le, K. Magaard and A. Paolini, \emph{Constructing characters of Sylow p-subgroups of finite Chevalley groups}, J. Algebra \textbf{468} (2016), 395--439. 


\bibitem[GMR15]{GMR15} S. M. Goodwin, P. Mosch and G. R\"ohrle, \emph{On the coadjoint orbits of maximal unipotent subgroups of reductive groups}, Transform. Groups \textbf{21} (2016), no.\ 2, 399--426.

\bibitem[GH97]{GH97} J. Gruber and G. Hiss, \emph{Decomposition numbers of finite classical groups for linear primes}, J. reine angew. Math. \textbf{485} (1997), 55--92.

\bibitem[His90]{His90} G. Hiss, \emph{Decomposition numbers of finite groups of Lie type in non-defining characteristic}, Progr. Math. \textbf{95}, Birkh\"auser, Basel (Bielefeld, 1991), 405--418.

\bibitem[His93]{His93} G. Hiss, \emph{Harish-Chandra series of Brauer characters in a finite group with a split $BN$-pair}, J. London Math. Soc. \textbf{48} (1993), 
no.\ 2, 219--228.

\bibitem[Him11]{Him11} F. Himstedt, \emph{On the decomposition numbers of the Ree groups $^2\mathrm{F}_4(q^2)$ in non-defining characteristic}, J. Algebra \textbf{325} (2011), no.\ 1, 364--403.

\bibitem[HLM11]{HLM11} F. Himstedt, T. Le and K. Magaard, {\em Characters of the Sylow $p$--subgroups of the Chevalley groups $D_4(p^n) $}, J.\ Algebra, {\bf 332} (2011), no.\ 1, 414--427.

\bibitem[HLM16]{HLM16} F. Himstedt, T. Le and K. Magaard, \emph{On the characters of the Sylow $p$-subgroups of untwisted Chevalley groups $\rY_n(p^a)$}, LMS J. Comput. Math. \textbf{19} (2016), 303--359.

\bibitem[HN14]{HN14} F. Himstedt and F. Noeske, \emph{Decomposition numbers of $\mathrm{SO}_7(q)$ and $\mathrm{Sp}_6(q)$}, J. Algebra \textbf{413} (2014), 15--40.

\bibitem[HN15]{HN15} F. Himstedt and F. Noeske, \emph{Restricting unipotent characters in special orthogonal groups}, LMS J. Comput. Math. \textbf{18} (2015) no.\ 1, 456--488.

\bibitem[Is]{Is} I. M. Isaacs, \emph{Character theory of finite groups}, Dover Books on Mathematics, New York, 1994.

\bibitem[Is07]{Is07} I. M. Isaacs, \emph{Counting characters of upper triangular groups}, J. Algebra \textbf{315} (2007), 698--719.

\bibitem[LM15]{LM15} T. Le and K. Magaard, {\em On the character degrees of Sylow $p$-subgroups of Chevalley groups $G(p^f)$ of type $E$}, Forum Math.\ {\bf 27} (2015), no.\ 1, 
1--55.

\bibitem[LMP]{LMP} T. Le, K. Magaard and A. Paolini, \emph{Computational details for the analysis of the nonabelian cores in types $\rD_6$ and $\rE_6$}, 
\url{http://www.mathematik.uni-kl.de/agag/mitglieder/wissenschaftliche-mitarbeiter/dr-alessandro-paolini/irreducible-characters-of-ud-6-q-and-ue-6-q/}. 

\bibitem[Leh74]{Leh74} G. Lehrer, \emph{Discrete series and the unipotent subgroup}, Compos. Math. \textbf{28} (1974), 9--19.

\bibitem[MT]{MT} G. Malle and D. Testerman, \emph{Linear algebraic groups and finite groups of Lie type},
Vol. 133, Cambridge University Press, 2011.

\bibitem[Pao17]{Pao17} A. Paolini, \emph{On the decomposition numbers of ${\rm SO}^+_8(2^f)$}, preprint, arXiv:1710.05473 (2017).

\bibitem[GAP4]{GAP4} The GAP Group, \emph{GAP -- Groups, Algorithms, and Programming, Version 4.8.6}, 2016, \url{http://www.gap-system.org}. 

\end{thebibliography}
\end{document}